\documentclass[11pt,a4paper, notitlepage]{amsart}
\advance\textwidth by 1.2in \advance\oddsidemargin by -.6in \advance\evensidemargin by -.6in
\usepackage[utf8]{inputenc}
\usepackage{colonequals}
\usepackage{graphicx,verbatim}
\usepackage{amssymb}
\usepackage{amsthm}
\usepackage{bbm}
\usepackage{amsmath}
\usepackage{mathtools}
\usepackage{textgreek}
\usepackage{mathbbol}
\usepackage{amsfonts}
\usepackage{mathrsfs}
\usepackage[english]{babel}
\usepackage{lscape} 
\usepackage[vcentermath,enableskew]{youngtab}
\usepackage{tikz,tikz-cd}
\usetikzlibrary{decorations.markings,intersections}
\usepackage{tabu}
\usetikzlibrary{arrows,positioning,automata,shadows,fit,shapes}
\usepackage{ stmaryrd }
\usepackage[english]{babel}
\usepackage{setspace}
\usepackage{microtype}
\usepackage[hang,flushmargin]{footmisc} 
\usepackage{ytableau}
\usepackage{hyperref}
\hypersetup{hypertexnames = false, bookmarksdepth = 2, bookmarksopen = true, colorlinks, linkcolor = black, citecolor = black, urlcolor = black, pdfstartview={XYZ null null 1}}

\newtheorem{theorem}{Theorem}[section]
\newtheorem{definition}[theorem]{Definition}
\newtheorem{example}[theorem]{Example}

\newtheorem{lemma}[theorem]{Lemma}
\newtheorem{proposition}[theorem]{Proposition}
\newtheorem{remark}[theorem]{Remark}
\newtheorem{conjecture}[theorem]{Conjecture}
\newtheorem*{theorem*}{Theorem}
\newtheorem*{definition*}{Definition}
\newtheorem*{lemma*}{Lemma}

\usepackage{chngcntr}
\counterwithin{table}{section}

\usepackage{tikz}
\usetikzlibrary{decorations.markings}
\usepackage{amsmath}
\usepackage{environ}

\NewEnviron{diagram}[2]{%
  \begin{tikzpicture}[
    baseline=(current bounding box.center),
    scale=1.2,
    every node/.style={circle, fill=black, inner sep=1pt},
    midarrow/.style={
      blue,
      thick,
      postaction={decorate},
      decoration={
        markings,
        mark=at position 0.5 with {\arrow{>}}
      }
    }
  ]
    \def\n{#1}

    \foreach \i in {1,...,\n} {
      \node (v\i) at ({360/\n * (\i - 1)}:1) {};
      \node[draw=none, fill=none] at ({360/\n * (\i - 1)}:1.2) {\(v_{\i}\)};
    }

    \foreach \i [evaluate=\i as \j using {int(mod(\i,\n)+1)}] in {1,...,\n} {
      \draw[dashed, gray] (v\i) -- (v\j);
    }

    #2
  \end{tikzpicture}
}

\tikzset{
  midarrow50/.style={
    blue, thick,
    postaction={decorate},
    decoration={markings, mark=at position 0.5 with {\arrow{>}}}
  },
  midarrow60/.style={
    blue, thick,
    postaction={decorate},
    decoration={markings, mark=at position 0.6 with {\arrow{>}}}
  },
  midarrow40/.style={
    blue, thick,
    postaction={decorate},
    decoration={markings, mark=at position 0.4 with {\arrow{>}}}
  }
}

\DeclareMathOperator{\End}{End}

\DeclareMathOperator{\ad}{ad}

\DeclareMathOperator{\Cent}{Cent}

\newcommand{\fso}{\mathfrak{so}}

\newcommand{\bbc}{\mathbb{C}}

\newcommand{\Cb}{\mathcal{B}}
\newcommand{\Cc}{\mathcal{C}}

\newcommand{\Cf}{\mathcal{F}}

\newcommand{\Ci}{\mathcal{I}}

\newcommand{\Cs}{\mathcal{S}}
\newcommand{\Ct}{\mathcal{T}}
\newcommand{\Cu}{\mathcal{U}}
\newcommand{\Cv}{\mathcal{V}}
\newcommand{\Cw}{\mathcal{W}}

\newcommand{\Ip}{\sigma}
\newcommand{\Sp}{\omega}
\newcommand{\Bp}{\beta}

\newcommand{\lpi}{\langle}
\newcommand{\rpi}{\rangle}

\newcommand{\ama}{\mathfrak{A}}

\newcommand{\tama}{\mathfrak{O}}

\newcommand{\clif}{\mathcal{C}}

\title{On angular momentum algebras and their relations}
\author{Kieran Calvert, Marcelo De Martino, Roy Oste}
\address{School of Mathematical Sciences\\ Lancaster University\\ Lancaster\\ LA1 4YW\\ United Kingdom}
\email[K.~Calvert]{kieran.calvert@lancaster.ac.uk}

\address{Forward College\\ Rua das Flores\\ 71, 1200-194 Lisboa\\ Portugal}
\email[M.~De Martino]{marcelo.demartino@forward-college.eu}

\address{Department of Mathematics, Computer science and Statistics\\ Ghent University\\
9000 Ghent\\
Belgium}
\email[R.~Oste]{roy.oste@ugent.be}
\date{August 2025}

\begin{document}

\begin{abstract}
    In this paper, we study the centraliser of $\mathfrak{osp}(1|2)$, denoted the total angular momentum algebra (TAMA), in the Weyl Clifford algebra. The TAMA extends the angular momentum algebra (AMA), which arises as the centraliser of
    \(\mathfrak{sl}(2)\) and admits a diagrammatic presentation via the crossing relation described by Feigin and Hakobyan. Using Young symmetrisers we construct an analogue relation for the even subalgebra of the TAMA. We prove that for rank $4$ and $5$ these relations generate a presentation for the even subalgebra of the TAMA.
\end{abstract}

\maketitle

\section{Introduction} 
The angular momentum algebra (AMA Definition \ref{def:AMA}) encodes the symmetries generated by orbital angular momentum operators in a wide variety of mathematical and physical contexts. It arises naturally as the centraliser algebra of the infinitesimal \(\mathfrak{sl}(2)\) action in the harmonic Howe dual pair \((O(n), \mathfrak{sl}(2))\), where the \(\mathfrak{sl}(2)\) is generated by the Laplacian, the squared norm, and the Euler operator. For more information on the topic of Howe dualities see \cite{GW09, Ho89, Ho89b, Ho95} and for this specific pair in the context of rational Cherednik algebras, see \cite{CDM22,CD20b}. Equivalently, the AMA can be realized as a quotient of the universal enveloping algebra \(\Cu(\mathfrak{so}(n))\) and the ideal generating the kernel of this quotient was elegantly described in the work of Feigin-Hakobyan \cite{FH15} in terms of crossing relations (Equation \ref{eq:CrossRels}) in the context of rational Cherednik algebras.

The total angular momentum algebra (TAMA Definition \ref{def:TAMA}) extends the AMA by incorporating spin operators alongside the orbital part, losely speaking forming a ``double cover'' of the AMA. Algebraically, this is achieved by enlarging the dual partner to include both the angular momentum and spin operators, which together close into a Lie superalgebra of type \(\mathfrak{osp}(1|2)\) in both the classical setting and the deformed setting of rational Cherednik algebras (see \cite{DBGV16,DBLROVJ22,DBOSS12,DBOJ18,DBOVJ18,DMRO23,LR21, LRO19,OSS09} and for a closely related unitary pair see \cite{BDSES10,CDBDMO20}). The TAMA arises as the centraliser algebra of this \(\mathfrak{osp}(1|2)\) action in the spinorial Howe dual pair \((\mathrm{Pin}(n), \mathfrak{osp}(1|2))\), with the polynomial-spinor space carrying the natural representation. Several structural properties of this algebra are known: a generating set was described in unpublished work of Oste \cite{Os21}, and the center of the TAMA was determined in \cite{CDMO24} in the more general deformed setting of rational Cherednik algebras.

In this paper, we explore analogies of the crossing relations of the AMA in the setting of the TAMA (Equation (\ref{e:relsinideal})). Our focus is on the even part of the TAMA, and our aim is to describe it in terms of generators and relations. We will work entirely in the classical, undeformed setting of the Weyl-Clifford algebra, but since the nature of the results obtained here are interpreted in the associated graded algebras, we expect that they can be generalized to the rational Cherednik setting in a rather natural way. 

The main novelties of this paper are the new relations Theorem~\ref{t:relsinideal} and Conjecture~\ref{con:maincon}, whose precise statements and proofs are contained in Section~\ref{Sec:crossingrelsforTAMA}.

We now give a breakdown of the contents of this paper.
In Section~\ref{sec:prelims} we introduce the Weyl and Clifford algebras with related bilinear pairings. Section~\ref{s:diagrbasisAMA}, in preparation for the main results, includes a proof of the diagrammatic basis given in~\cite{FH15}. In Section~\ref{s:EvenTAMA} we show that there is a natural homomorphism from the universal enveloping algebra $\mathcal{U}(\mathfrak{so}(n))$, corresponding to the infinitesimal action of the $\mathrm{Pin}(n)$ group, onto the even part of the TAMA. Section~\ref{s:TableauxRels} reinterprets the crossing relation of the AMA as a Young symmetriser of suitable tableaux and describes its natural analogue in the TAMA, leading to the candidate ideal of relations \eqref{e:relsinideal}. In Section~\ref{Sec:crossingrelsforTAMA} we prove in Theorem~\ref{t:relsinideal} that these relations hold in the even TAMA for all $n$. We conclude with Conjecture~\ref{con:maincon}, which states that this ideal is the full ideal of relations for the even TAMA for arbitrary $n$. Finally, in Section~\ref{sec:lastsection} we introduce diagrams associated to monomials in the AMA-Clifford algebra and the even TAMA algebra (Definitions~\ref{d:AMAEdiags} and~\ref{def:TAMAdiagram}) and use them to prove linear dependence for a set of monomials we call \emph{uncrossable monomials}, from which the main conjecture follows in the cases $n=4$ and $n=5$.

\subsection*{Acknowledgments}
We gratefully acknowledge the hospitality and excellent working conditions provided by the Mathematisches Forschungsinstitut Oberwolfach, where KC and MDM were Oberwolfach Research Fellows during the Spring of 2024. MDM thanks the Mathematics Department at Lancaster University for hospitality during a visit and collaboration with KC. MDM acknowledges support from the special research fund (BOF) of Ghent University [BOF20/PDO/058].

\section{Preliminaries} \label{sec:prelims}

Let \((V_0, \sigma_0)\) be a real inner product space. Denote by \((V, \Ip)\) its complexification, where \(\sigma\) is a nondegenerate, symmetric bilinear form on \(V\). Let \(\langle \cdot, \cdot \rangle : V \times V^* \to \mathbb{C}\) be the natural bilinear pairing. Denote by \(\omega\) the natural extension of \(\langle \cdot, \cdot \rangle\) to a symplectic pairing on the vector space \(\Cv = V \oplus V^*\). Specifically, for all \(y, y' \in V\) and \(x, x' \in V^*\), the pairing \(\omega\) satisfies the identities:
\[
\omega(y, y') = 0, \quad \omega(x, x') = 0, \quad \text{and} \quad \omega(y, x) = \langle y, x \rangle = - \omega(x, y).
\]
We shall fix, once and for all, a basis \(\Cb = \{y_1, \dots, y_n, x_1, \dots, x_n\}\) of \(\Cv\) satisfying the properties that \(\{y_i\} \subset V\), \(\{x_j\} \subset V^*\), and 
\[
\langle y_i, x_j \rangle = \delta_{ij} = \omega(y_i, x_j) = \sigma(y_i, y_j).
\]
The pairs \((V, \Ip)\) and \((\Cv, \Sp)\) induce quadratic algebras, which will be central to this paper. We recall their definitions in the following subsections.

\subsection{Weyl Algebra}\label{sec:WeylAlgebra}

Recall that, associated with the symplectic pair \((\Cv, \omega)\), we have the Weyl algebra \(\Cw = \Cw(\Cv, \omega)\), which is the associative algebra defined as the quotient  
\[
\Cw = \Ct(\Cv)/\Ci(\omega),
\]
where \(\Ct(\Cv) = \bigoplus_{m \geq 0} \Ct^m(\Cv)\) is the tensor algebra and \(\Ci(\omega)\) is the two-sided ideal generated by the set
\[
\{v \otimes v' - v' \otimes v - \omega(v, v') \mid v, v' \in \Cv\}.
\]
Equivalently, \(\Cw\) is the complex associative algebra with unit \(1\), generated by the basis \(\Cb\) and subject to the relations
\begin{equation}\label{eq:WeylRels}
    [v, v'] = \omega(v, v')1
\end{equation}
for all \(v, v' \in \Cb\), where $[\cdot,\cdot]$ denotes the commutator.

The Weyl algebra is naturally equipped with a filtration \(\Cf^\bullet(\Cw)\), where, for each integer \(m \geq 0\), the space \(\Cf^m(\Cw)\) is spanned by all products \(v_1 \cdots v_p\) with \(p \leq m\) elements in \(\Cv\). It is clear that \(\Cf^p(\Cw) \Cf^q(\Cw) \subseteq \Cf^{p+q}(\Cw)\). Moreover, from the defining relation \eqref{eq:WeylRels}, it follows that the associated graded algebra
\[
\mathsf{gr}^\bullet(\Cw) = \bigoplus_{m \geq 0} \mathsf{gr}^m(\Cw)
\]
is isomorphic to the symmetric algebra \(\Cs^\bullet(\Cv)\). This isomorphism can be explicitly described via the quantization map \(Q_w: \Cs^\bullet(\Cv) \to \Cw\), defined by
\begin{equation}\label{eq:WeylQuant}
    Q_w(v_1 \cdots v_m) = \frac{1}{m!} \sum_{\sigma \in S_m} v_{\sigma(1)} \cdots v_{\sigma(m)},
\end{equation}
for all \(v_1, \dots, v_m \in \Cv\). As a consequence of the linear isomorphism \(\Cs^\bullet(\Cv) \cong \Cw\), the Weyl algebra admits a linear basis consisting of monomials, which can be expressed in multi-index notation as
\[
y^\alpha x^\beta = y_1^{\alpha_1} \cdots y_n^{\alpha_n} x_1^{\beta_1} \cdots x_n^{\beta_n},
\]
where \(\alpha = (\alpha_1, \dots, \alpha_n)\) and \(\beta = (\beta_1, \dots, \beta_n)\) are elements of \(\mathbb{Z}_{\geq 0}^n\).

\subsubsection{Kostant pairing}

Following Kostant \cite{K01}, we extend the symplectic pairing \(\omega\) on \(\Cv = V \oplus V^*\) to a nondegenerate bilinear form \(\kappa = \kappa(\omega)\) on the symmetric algebra \(S^\bullet(\Cv)\). This extension satisfies the orthogonality condition \(S^p(\Cv) \perp S^q(\Cv)\) for \(p \neq q\) and is otherwise given by
\begin{equation}\label{eq:KostantPairing}
    \kappa(u_1 \cdots u_p, v_1 \cdots v_p) = \sum_{\pi \in S_p} \omega(u_1, v_{\pi(1)}) \omega(u_2, v_{\pi(2)}) \cdots \omega(u_p, v_{\pi(p)})
\end{equation}
for homogeneous elements \(u_1 \cdots u_p, v_1 \cdots v_p \in \Cs^p(\Cv)\). We refer to \(\kappa\) as the {\it Kostant pairing}.  Using the quantization isomorphism, this induces a nondegenerate bilinear form on \(\Cw\). For simplicity, we use the same notation for the pairing in both \(\Cw\) and \(\Cs^\bullet(\Cv)\).

Given a countably infinite-dimensional vector space \(\Cu\) equipped with a nondegenerate bilinear form \(\Bp\), we say that a basis \(\Cb = \{u_1, u_2, \dots\}\) is an {\it orthogonal basis for \(\Bp\)} if, for each \(u_i \in \Cb\), there exists a {\it unique} \(u_j \in \Cb\) such that \(\Bp(u_i, u_j) \neq 0\). Additionally, for a multi-index \(\alpha = (\alpha_1, \dots, \alpha_n) \in \mathbb{Z}_{\geq 0}^n\), we use the standard notations:
\[
x^\alpha = x_1^{\alpha_1} \cdots x_n^{\alpha_n}, \quad |\alpha| = \alpha_1 + \cdots + \alpha_n, \quad \text{and} \quad \alpha! = \alpha_1! \alpha_2! \cdots \alpha_n!.
\]
The following result is well known.

\begin{proposition}
    With respect to the Kostant pairing, the monomial basis \(\{x^\alpha y^\beta \mid \alpha, \beta \in \mathbb{Z}_{\geq 0}^n\}\) is an orthogonal basis of \(\Cw\).
\end{proposition}
\begin{proof}
    For multi-indices \(\alpha, \beta, \mu, \nu \in \mathbb{Z}_{\geq 0}^n\), we compute
    \[
    \kappa(x^\alpha y^\beta, x^\mu y^\nu) = (-1)^{|\alpha|} \delta_{\alpha, \nu} \delta_{\beta, \mu} \alpha! \beta!,
    \]
    which proves the desired claim.
\end{proof}

\subsection{Clifford Algebra}
Focusing on the quadratic pair \((V, \Ip)\), we define the Clifford algebra \(\Cc = \Cc(V, \Ip)\) as the associative algebra given by the quotient  
\[
\Cc = \Ct(V)/\Ci(\Ip),
\]
where \(\Ct(V)\) is the tensor algebra, and the ideal \(\Ci(\Ip)\) is generated by the set\footnote{Following standard conventions, we include the factor \(2\) in the definition of the ideal, despite its asymmetry with the convention used in the Weyl algebra.}
\[
\{y \otimes y' + y' \otimes y - 2\Ip(y, y') \mid y, y' \in V\}.
\]
Equivalently, following standard conventions, let \(\iota: V \to \Cc(V, \Ip) = \Ct(V)/\Ci(\Ip)\) denote the embedding of \(V\) into the Clifford algebra. Then, each basis element \(y_j\) of \(V\) is mapped to a generator \(e_j = \iota(y_j)\) and \(\Cc\) is the complex associative algebra with unit \(1\), generated by \(\{e_j\}\) and subject to the relations
\begin{equation}\label{eq:CliffRels}
    \{e_i, e_j\} = 2\Ip(y_i, y_j)1,
\end{equation}
where \(\{\cdot,\cdot\}\) denotes the anti-commutator. 

The algebra \(\Cc\) is naturally filtered in terms of the generators \(\{e_i\}\), and its associated graded algebra is isomorphic to the exterior algebra \(\bigwedge^\bullet(V)\). The quantization isomorphism  
\(
Q_c: \bigwedge^\bullet(V) \to \Cc
\)
is given by
\begin{equation}\label{eq:CliffQuant}
    Q_c(v_1 \wedge \cdots \wedge v_m) = \frac{1}{m!} \sum_{\pi \in S_m} (-1)^\pi \iota(v_{\pi(1)}) \cdots \iota(v_{\pi(m)}),
\end{equation}
for all \(v_1, \dots, v_m \in V\). 

As with the Weyl algebra, the linear isomorphism \(\bigwedge^\bullet(V) \cong \Cc\) implies that the Clifford algebra admits a linear basis consisting of monomials. These are expressed in multi-index notation as
\[
e^\gamma = e_1^{\gamma_1} \cdots e_n^{\gamma_n}, 
\]
where \(\gamma = (\gamma_1, \dots, \gamma_n) \in \{0,1\}^n\). Occasionally, we may refer to a product of generators in \(\Cc\) by 
\[
e_A = e_{a_1}e_{a_2}\cdots e_{a_k},
\]
associated to a sequence $A = ( a_1,\ldots, a_k)$  of length $k \leq n$ of distinct integers between $1$ and $n$.

Finally, by defining \(\Cc_{\overline{0}}\) (respectively \(\Cc_{\overline{1}}\)) as the span of all those \(e^\gamma\) such that \(|\gamma| = \sum_{i=1}^n\gamma_i\) is even (respectively odd), from the Clifford relation \eqref{eq:CliffRels}, we have the well-defined \(\mathbb{Z}_2\)-grading \(\Cc = \Cc_{\overline{0}} \oplus \Cc_{\overline{1}}\).

\subsubsection{Determinant pairing}

Similarly to the Kostant pairing, given the quadratic pair \((V, \Ip)\), there is a natural extension of \(\Ip\) to a nondegenerate bilinear pairing  \(\delta = \delta(\Ip),\)  
called the {\it determinant pairing}. This pairing satisfies \(\bigwedge^p(V) \perp \bigwedge^q(V)\) for \(p \neq q\) and is otherwise given by  
\begin{equation}\label{eq:DetPairing}
    \delta(u_1 \wedge \cdots \wedge u_p, v_1 \wedge \cdots \wedge v_p) = \sum_{\pi \in S_p} (-1)^\pi \Ip( u_1, v_{\pi(1)}) \Ip( u_2, v_{\pi(2)}) \cdots \Ip( u_p, v_{\pi(p)})
\end{equation}  
on homogeneous elements \(u_1 \wedge \cdots \wedge u_p, v_1 \wedge \cdots \wedge v_p \in \bigwedge^p(V)\).  

Using the quantization isomorphism, this induces a nondegenerate bilinear pairing on \(\Cc\), which we also denote by \(\delta\). We observe that the expression in \eqref{eq:DetPairing} corresponds to the {\it determinant} of the matrix \(\big(\Ip(u_i, v_j)\big)_{i,j}\), whereas \eqref{eq:KostantPairing} can be rephrased as the {\it permanent} of the matrix \(\big(\Sp(u_i, v_j)\big)_{i,j}\).  

It is well known that the (finite) linear basis  
\(
\{e^\gamma \mid \gamma \in \{0,1\}^n\}
\)
of \(\Cc\) is orthogonal with respect to the determinant pairing \(\delta\). 

\subsection{The Weyl-Clifford algebra}\label{sec:WeylCliff}

In the previous subsections, we described the algebras \(\Cw = \Cw(\Cv, \Sp)\) and \(\Cc = \Cc(V, \Ip)\), which arise naturally from a bilinear pairing on a vector space. We now define the Weyl-Clifford algebra as the tensor product over \(\mathbb{C}\) of these algebras:  
\[
\Cw\Cc = \Cw \otimes \Cc.
\]
Note that the \(\mathbb{Z}_2\)-grading of \(\Cc\) induces a likewise grading on the Weyl-Clifford algebra \(\Cw\Cc\).

It is important to note that this algebra depends entirely on \((V, \Ip)\), since from this pair we obtained the symplectic pair \((\Cv, \Sp)\), and from these structures, we constructed both factors of \(\Cw\Cc\). From what has been described so far, it follows that the set  
\[
\{x^\alpha y^\beta\otimes e^\gamma\mid \alpha, \beta\in\mathbb{Z}_{\geq 0}^n, \gamma\in\{0,1\}^n\}
\]  
forms a linear basis of \(\Cw\Cc\). Furthermore, the bilinear pairing  
\(
\Bp = \kappa \otimes \delta
\)
defined by  
\[
\Bp(x^\alpha y^\beta\otimes e^\gamma, x^{\alpha'} y^{\beta'}\otimes e^{\gamma'}) = \kappa(x^\alpha y^\beta, x^{\alpha'} y^{\beta'}) \delta(e^\gamma, e^{\gamma'}),
\]
with $\alpha,\alpha',\beta,\beta'\in\mathbb{Z}_{\geq 0}^n$ and $\gamma,\gamma'\in\{0,1\}^n$ is a nondegenerate pairing on \(\Cw\Cc\). 

\section{Diagrammatic Basis for the Angular Momentum Algebra} \label{s:diagrbasisAMA}

We begin by describing a subalgebra of \(\Cw\), known as the {\it angular momentum algebra} (or simply AMA), which arises as a homomorphic image of the universal enveloping algebra \(\Cu(\mathfrak{so}(n, \mathbb{C}))\).  In this paper, we use the definition \(\mathfrak{so}(n, \mathbb{C}) = \{X \in \End(\bbc^n)\mid X^T+X=0\}\).

\begin{definition} \label{def:AMA}
Let \(\ama\) denote the associative, unital subalgebra of \(\Cw\) generated by the angular momentum operators, defined for all \(1 \leq i < j \leq n\) as  
\begin{equation}\label{eq:AngMomOp}
L_{ij} := x_i y_j - x_j y_i.    
\end{equation}
\end{definition}

It is well known that the set \(\{L_{ij} \mid 1 \leq i < j \leq n\}\) spans a subspace of \(\Cw\) that is closed under the commutator in \(\Cw\) and is isomorphic to the Lie algebra \(\mathfrak{so}(n, \mathbb{C})\).  Hence, the assignment \(X_{ij}:=E_{ij}-E_{ji}\mapsto L_{ij}\) yields a surjective homomorphism  
\[
\Cu(\mathfrak{so}(n, \mathbb{C})) \to \ama,
\]
whose kernel contains the non-homogeneous quadratic ideal of relations generated by  
\begin{equation}\label{eq:CrossRels}
X_{ij}X_{kl} + X_{ik}X_{lj} + X_{il}X_{jk} = 
X_{ij} \delta_{kl} + X_{ik} \delta_{lj} + X_{il} \delta_{jk}   
\end{equation}
for all \(1 \leq i,j,k,l \leq n\). We can further characterize \(\ama\) as the subalgebra of \(\Cw\) that centralizes the harmonic \(\mathfrak{sl}(2)\) realized by the Laplace operator and its dual, multiplication by the norm-square operator (see Ciubotaru-De Martino and Weyl).

Following Feigin-Hakobyan, the relations in \eqref{eq:CrossRels} are referred to as the {\it crossing relations}. In this section, we provide a detailed proof that the ideal generated by \eqref{eq:CrossRels} coincides with the kernel of the homomorphism \(\Cu(\mathfrak{so}(n, \mathbb{C})) \to \ama\). While this result, as well as its generalization in the setting of the rational Cherednik algebra, is known from \cite{FH15} and other sources, we have chosen to include a proof as preparation for the main novelties in the context of the Weyl-Clifford algebra and, more generally, the rational Cherednik-Clifford algebra.

\subsection{Non-crossing diagrams}

Let \(\Pi = \Pi_n = \{(i,j) \in \mathbb{Z}_{\geq 0}^2 \mid 1 \leq i \neq j \leq n\}\). We order the set \(\Pi\) lexicographically and write \((i_1, j_1) \preceq (i_2, j_2)\) to denote this total ordering.  Now consider a regular \(n\)-gon inscribed in the unit circle. For definiteness, we place and label its \(n\) vertices as \(v_j\), where \(j=1,\ldots,n\), at the points  
\(
v_j = e^{2\pi \sqrt{-1} (j-1)/n}
\)  
in the plane.
We interpret a chord \(c = (i, j) \in \Pi\) as an oriented straight line segment starting at \(v_i\) and ending at \(v_j\). Furthermore, to any such $c=(i,j)\in \Pi$, we will naturally associated with it the angular momentum operator \(L_c = L_{ij} \in \ama\) of \eqref{eq:AngMomOp}.

\begin{definition}\label{def:Diagrams}
    A {\it diagram} is a sequence \(D = (c_1, c_2, \ldots, c_p)\) of chords \(c_s \in \Pi\) for \(1 \leq s \leq p\) such that if \(1 \leq s < t \leq p\), then \(c_s \preceq c_t\). The positive integer \(p = |D|\) is called the {\it length} of \(D\). Denote by \(\mathscr{D}_p\) the set of all diagrams of length \(p\) and by \(\mathscr{D}\) the set of all diagrams.
\end{definition}

\begin{figure}[ht]
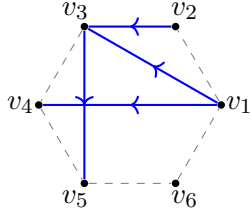

\centering
  \begin{diagram}{6}{
    \draw[midarrow] (v1) -- (v4);
    \draw[midarrow] (v1) -- (v3);
    \draw[midarrow] (v2) -- (v3);
    \draw[midarrow] (v3) -- (v5);
  }
  \end{diagram}
  \caption{Diagram \(((1,3),(1,4),(2,3),(3,5))\) with \(n = 6\) variables.}
  \label{fig:hexdiagram}
\end{figure}

From the definition it follows that \(\mathscr{D} = \bigcup_{p \geq 0} \mathscr{D}_p\).  
Given a diagram \(D = (c_1, c_2, \ldots, c_p) \in \mathscr{D}_p\) with \(c_s = (i_s, j_s) \in \Pi\), we associate to \(D\) an element $L_D\in\ama$ by taking the ordered product
\begin{equation}\label{eq:DiagramMonAMA}
    L_D := L_{c_1} L_{c_2} \cdots L_{c_p} = \prod_{s=1}^p L_{i_s j_s}.
\end{equation}
Furthermore, such diagram \(D\) determines sequences 
\[
I(D) := (i_1, i_2, \ldots, i_p) \quad \text{and} \quad J(D) := (j_1, j_2, \ldots, j_p),
\]
recording respectively the initial and terminal vertices of the chords. By construction, the lexicographic ordering \(c_1 \preceq c_2 \preceq \ldots \preceq c_p\) ensures that the sequence  \(I(D)\) is (weakly) ordered, meaning that \(s < t\) implies \(1 \leq i_s \leq i_t \leq n\). However, this ordering does not necessarily hold for \(J(D)\).  

To capture the combinatorial data of \(D\) in a more algebraic form, we define for each vertex \(q\in\{1,\ldots,n\}\) the multiplicities
\[ 
\alpha_q := \textup{mult}_q(I(D)) \quad \text{and} \quad \beta_q := \textup{mult}_q(J(D)),
\]
that is, \(\alpha_q\) counts how many times \(q\) appears as an initial vertex, and \(\beta_q\) counts how many times \(q\) appears as a terminal vertex.
\begin{definition}
    The {\it associated monomial} \(m_D \in \Cw\) of a diagram \(D\) is given by  
    \[
    m_D := x^{\alpha(D)} y^{\beta(D)},
    \]
    where \(\alpha(D) = (\alpha_1, \ldots, \alpha_n)\) and \(\beta(D) = (\beta_1, \ldots, \beta_n)\).
\end{definition}

Although the assignment \(\mathscr{D} \ni D \mapsto m_D\) is well-defined, it is not necessarily injective. For instance, the diagrams \(D = ((1,3), (2,4))\) and \(D' = ((1,4), (2,3))\) both have  
\[
m_D = x_1 x_2 y_3 y_4 = m_{D'}
\]
as their associated monomial.  This non-uniqueness arises because the combinatorial data recorded by \(m_D\) does not capture how the chords are positioned or whether they cross each other. To address this, we define a subset of diagrams uniquely determined by their associated monomials. Partition \(\Pi = \Pi_+ \cup \Pi_-\), where \((i,j) \in \Pi_+\) if and only if \(i < j\), and let \(\Pi_-\) be the complement.  

\begin{definition}
  A diagram \(D \in \mathscr{D}\) is called a {\it non-crossing diagram} if \(c_s \in \Pi_+\) for all \(1 \leq s \leq |D|\) and if for all \(1 \leq s,t \leq |D|\), the condition
\begin{equation}\label{eq:NonCross}
    i_s < i_t < j_s \quad\Rightarrow\quad  j_t \leq j_s
\end{equation}
is satisfied. Denote by \(\mathscr{D}_+\) the set of all non-crossing diagrams.  
Conversely, we say a diagram is crossed if there exists $1 \leq s,t \leq |D|$ such that 
\begin{equation}\label{eq:Cross}
    i_s < i_t < j_s \quad\text{and}\quad  j_t > j_s
\end{equation}.
\end{definition}

\begin{remark}\label{rem:CrossingRels}
    If we assume \(i < j < k < l\), the relation \eqref{eq:CrossRels} implies that, in \(\ama\),  
    \[
    L_{ik} L_{jl} \equiv L_{ij} L_{kl} + L_{il} L_{jk}
    \]
    modulo lower-order terms in \(\Cw\). Intuitively, this means that crossing chords \(((i,k), (j,l))\) can be rewritten as a sum of two terms corresponding to non-crossing diagrams. Specifically, the sequences \(((i,j), (k,l))\) and \(((i,l), (j,k))\) satisfy \eqref{eq:NonCross}. This will be made precise later, but we depict this situation in Figure \ref{fig:Uncrossing}, below. 
\end{remark}

\begin{figure}[ht]
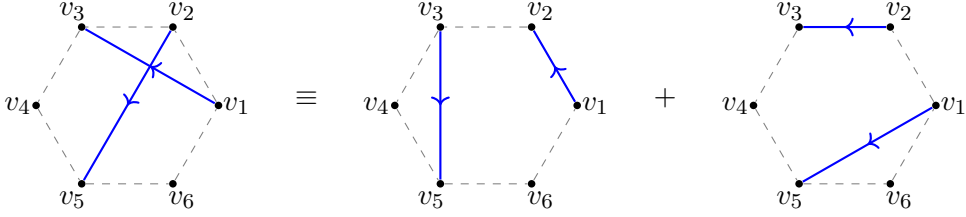

\centering
  \begin{diagram}{6}{
    \draw[midarrow] (v1) -- (v3);
    \draw[midarrow] (v2) -- (v5);
  }
  \end{diagram} \quad \(\equiv\) \quad
  \begin{diagram}{6}{
    \draw[midarrow] (v1) -- (v2);
    \draw[midarrow] (v3) -- (v5);
  }
  \end{diagram}\quad + \quad
  \begin{diagram}{6}{
    \draw[midarrow] (v1) -- (v5);
    \draw[midarrow] (v2) -- (v3);
  }
  \end{diagram}
  \caption{A crossing diagram and its uncrossed factors for \((i,j,k,l)=(1,2,3,5).\)}
  \label{fig:Uncrossing}
\end{figure}

\begin{proposition}\label{prop:Injective}
    The assignment \(\mathscr{D}_+ \ni D \mapsto m_D \in \Cw\) is injective.  
\end{proposition}

\begin{proof}
    Let \(D = (c_1, \ldots, c_p)\) and \(D' = (c_1', \dots, c_q')\) be non-crossing diagrams. Denote by \(I = (i_1, \ldots, i_p)\) and \(I' = (i'_1, \ldots, i'_q)\) their respective sequences of initial points. Similarly, let \(J\) and \(J'\) be their sequences of endpoints. The assumption \(m_D = m_{D'}\) implies that \(I = I'\) and \(J = \pi(J')\) for some permutation \(\pi \in S_q\). In particular, this implies \(|D| = p = q = |D'|\).  

    We prove by induction on \(p = |D|\) that for all \(D, D' \in \mathscr{D}_p \cap \mathscr{D}_+\), the condition \(m_D = m_{D'}\) implies \(D = D'\). If \(p = 1\), so \(D = ((i,j))\) and \(D' = ((i',j'))\), the claim is immediate.  

    Now, suppose by induction that for any \(1 < |\hat{D}| = |\hat{D}'| < p\), the assumption \(m_{\hat{D}} = m_{\hat{D}'}\) implies \(\hat{D} = \hat{D}'\). Consider diagrams \(D, D' \in \mathscr{D}_+\) with \(|D| = p = |D'|\) and assume \(m_D = m_{D'}\).   Let \(j = \min(J) = \min(J')\) and let \(s, t\) be indices in \(J\) and \(J'\), respectively, such that \(j_s = j = j'_t\). 
    
    If \(s = t\), then \(i_s = i'_t = i\), and removing the chords \((i,j)\) from both diagrams yields \(\hat{D}, \hat{D'}\) with fewer chords satisfying \(m_{\hat{D}} = m_{\hat{D}'}\), so \(D = D'\) follows by induction. 
    
    If, on the other hand, \(s \neq t\), without loss of generality, we can assume \(s<t\) so that the ordering of \(I=I'\) implies \(i_s\leq i'_t\). Since \(c'_t=(i'_t,j)\in \Pi_+\), we have
    \[
    i_s \leq i'_t=i_t < j=j_s=j'_t.
    \]
Now focus on the pair \(c_t = (i_t,j_t)\succeq c_s = (i_s,j)\). Since \(j=\min(J)\), we have \(j_t\geq j\). If it were the case that both \(j_t > j\) and \(i_s<i'_t\) we would have a crossing since \(i_s<i_t<j_s\) and yet \(j_t>j_s=j\) contradicting \eqref{eq:NonCross}. Hence, \(j_t=j\) or \(i_s = i'_t\). In any case, we can apply the inductive hypothesis to the diagrams \(\hat{D}\) and \(\hat{D'}\) obtained by removing the chords \(c_t\), \(c'_t\) or \(c_s\), \(c'_t\), from which \(D=D'\), finishing the proof.
\end{proof}

\begin{remark}
    The condition that each \(c_s \in \Pi_+\) is crucial for injectivity. If arbitrary orientations were allowed, distinct diagrams such as \(D = ((1,4), (3,2))\) and \(D' = ((1,2), (3,4))\) (see Figure \ref{fig:SameMonomial}) would not have crossings and yet satisfy \(m_D = m_{D'}\).
\end{remark}

\begin{figure}[ht]
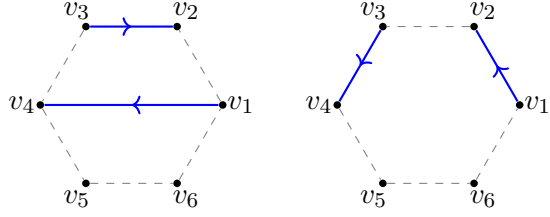

\centering
  \begin{diagram}{6}{
    \draw[midarrow] (v1) -- (v4);
    \draw[midarrow] (v3) -- (v2);
  }
  \end{diagram} \quad 
  \begin{diagram}{6}{
    \draw[midarrow] (v1) -- (v2);
    \draw[midarrow] (v3) -- (v4);
  }
  \end{diagram}
  \caption{Distinct uncrossed diagrams with same monomial \(x_1x_3y_2y_4\).}
  \label{fig:SameMonomial}
\end{figure}

\subsection{Kernel of the homomorphism onto the AMA}

Our goal is to show that the ideal generated by \eqref{eq:CrossRels} coincides with the kernel of the homomorphism \(\Cu(\mathfrak{so}(n,\mathbb{C}))\to\ama\). Since \(\ama\) is generated by the angular momentum operators, the set 
\[
\{L_D \mid D \in \mathscr{D} \}
\]
contains a linear basis for \(\ama\). Since \(L_{ji} = -L_{ij}\) and by repeatedly applying \eqref{eq:CrossRels}, any element \( L_D \) with an arbitrary \( D \in \mathscr{D}_p \) can be expressed as a linear combination of elements in \( \{L_{D'}\} \), where \( D' \in \mathscr{D}_p \cap \mathscr{D}_+ \) or \( |D'| < p \). Consequently, the set  
\[
\{L_D \mid D \in \mathscr{D}_+\}
\]
spans \(\ama\) as a vector space. We now argue that this set forms a linear basis for \(\ama\), which requires proving that \(\{L_D \mid D \in \mathscr{D}_+\}\) is linearly independent.

Before proceeding, we establish additional properties of the sets \(\Pi\) and \(\mathscr{D}\). The partition \(\Pi = \Pi_+ \cup \Pi_-\) naturally defines an involution on \(\Pi\) by reversing chords: for any \( c = (i,j) \in \Pi \), we define \( \overline{c} = (j,i) \). Clearly, if \( c \in \Pi_+ \) with \( i < j \), then the reversed chord satisfies \( \overline{c} \succeq c \) in the lexicographic order of \(\Pi\). 

Since \( c \preceq c' \) defines a total order on \(\Pi\), we extend the lexicographic order to the set \(\mathscr{D}\) of all diagrams, denoting it by \( D \preceq D' \). Naturally, this order satisfies \( D \prec D' \) if \( |D| < |D'| \). Moreover, since each set \( \mathscr{D}_p = \{D \in \mathscr{D} \mid |D| = p\} \) is finite for all \( p \in \mathbb{Z}_{\geq 0} \), it follows that for any given \( D \in \mathscr{D}_p \), there exist only finitely many \( D' \in \mathscr{D}_p \) satisfying \( D' \succ D \).

\begin{proposition}\label{prop:Triangular}
For each \( D \in \mathscr{D}_+ \) with \( |D| = p \in \mathbb{Z}_{> 0} \), we can write
\begin{equation}\label{eq:ChangeBasis}
    L_D \equiv m_D + \sum_{D' \succ D} \lambda_{D'} m_{D'},
\end{equation}
modulo lower order terms in \(\Cw\), where \( \lambda_{D'} \in \mathbb{Z} \) and the sum runs over \( D' \in \mathscr{D}_p \).
\end{proposition}

\begin{proof}
Whenever \( c = (i,j) \in \Pi \), let us use the shorthand \( m_c \) instead of \( m_{(c)} \) for the monomial associated with the diagram \( (c) \). If \( c \in \Pi_+ \), then \( L_c = m_c + m_{\overline{c}} \), and since \( c = (i,j) \) with \( i < j \), it follows that \( \overline{c} \succ c \).

Now, given \( D = (c_1, c_2, \ldots, c_p) \in \mathscr{D}_+ \), we have:
\[
L_D = (m_{c_1} - m_{\overline{c_1}})(m_{c_2} - m_{\overline{c_2}}) \cdots (m_{c_p} - m_{\overline{c_p}}).
\]
Thus, the diagrams \( D' \) appearing in the sum are obtained by applying  nontrivial elements of the group  \( C_2^p  \) generated by the chord-reversing involutions to the diagram \( D \). Since reversing any chord \( c_s \mapsto \overline{c_s} \) results in a diagram \( D' \) with \( D' \succ D \), it follows that expanding the product above leads to terms \( D' \) satisfying \( D \prec D'\). This proves equation \eqref{eq:ChangeBasis}.
\end{proof}

\begin{theorem}[\emph{cf.} \cite{FH15}]
The set \( \{L_D \mid D \in \mathscr{D}_+\} \subset \ama \) is linearly independent.
\end{theorem}

\begin{proof}
Proposition \ref{prop:Injective} establishes a well-defined subset \( \{m_D \mid D \in \mathscr{D}_+\} \) within the monomial basis \( \{x^\alpha y^\beta \mid \alpha, \beta \in \mathbb{Z}_{\geq 0}^n\} \) of \(\Cw\). Since the latter is a basis, the set \( \{m_D \mid D \in \mathscr{D}_+\} \) is linearly independent. 

From Proposition \ref{prop:Triangular}, it follows that \( \{L_D \mid D \in \mathscr{D}_+\} \) is also linearly independent. To see this, consider the subset \( \{L_D \mid D \in \mathscr{D}_+ \cap \mathscr{D}_p\} \) for each \( p \). Enumerate \( \mathscr{D}_+ \cap \mathscr{D}_p \) lexicographically as \( D_1 \prec D_2 \prec \cdots \prec D_M \), where \( M = |\mathscr{D}_+ \cap \mathscr{D}_p| \). By Proposition \ref{prop:Triangular}, each \( L_{D_s} \) can be written as
\[
L_{D_s} = m_{D_s} + \varepsilon_{D_s},
\]
where \( \varepsilon_{D_s} \) is a linear combination of monomials of lower filtration order or corresponding to diagrams greater than \( D_s \) in the lexicographic order.

Now, if \(\sum_s \lambda_s L_{D_s} = 0 = \sum_s \lambda_s (m_{D_s} + \varepsilon_{D_s})\), starting with the smallest index, from Proposition \ref{prop:Triangular}, there can be no linear dependence between \(m_{D_1}\) and a summand of 
\(\varepsilon_{D_s}\) for \(s>1\) which implies \(\lambda_1=0\). Proceeding by induction, the result follows.
\end{proof}

\section{Analog of crossing relations for the total angular momentum algebra} \label{Sec:crossingrelsforTAMA}

In the previous section, we discussed the AMA, a subalgebra of \(\Cw\) that arises as the centraliser algebra of the harmonic \(\mathfrak{sl}(2)\) subalgebra of the Weyl algebra that contains the Laplace operator. It serves as the infinitesimal realization of the orthogonal group appearing in the harmonic Howe dual pair \((O(n),\mathfrak{sl}(2))\) within the Weyl algebra. 

In this section, we focus on the related dual pair \((\mathrm{Pin}(n),\mathfrak{spo}(2|1))\), which is realized in the Weyl-Clifford algebra. The harmonic \(\mathfrak{spo}(2|1)\) Lie algebra is generated by the operators
\[
\partial_{\underline{x}} = \sum_j y_j \otimes e_j, \qquad
\underline{x} = \sum_j x_j \otimes e_j,
\]
where \(\partial_{\underline{x}}\) is the well-known Dirac operator, which acts as an operator square root of the Laplacian (add references). This algebra comes equipped with a natural \(\mathbb{Z}_2\)-grading induced by the grading on \(\Cw\Cc\), that is, \(\mathfrak{spo}(2|1)_{\overline \jmath} = \mathfrak{spo}(2|1) \cap \Cw\Cc_{\overline \jmath}\), with \({\overline \jmath} \in \{\overline 0, \overline 1\}\).

\begin{definition}\label{def:TAMA}
    The subalgebra of \(\Cw\Cc\) that centralizes the harmonic \(\mathfrak{spo}(2|1)\) Lie superalgebra is called the Total Angular Momentum Algebra, or TAMA, in short. It is denoted by \(\tama\). This is a \(\mathbb{Z}_2\)-graded associative algebra whose grading is inherited from the \(\mathbb{Z}_2\)-grading of \(\Cc\).
\end{definition}

\subsection{Generators of the TAMA and the even subalgebra}\label{s:EvenTAMA}
We now briefly describe what is known about the structure of the algebra $\tama$. Consider the extremal projector 
\[
P = 1 - \frac{1}{2}\partial_{\underline{x}}\underline{x} \in \Cw\Cc
\]
and let \(\ad_P\) denote the adjoint action (via the graded commutator) of \(P\) on \(\tama\). It is clear that \(\ad_P(a) = a\) for any element \(a\in\tama\), since \(\partial_{\underline{x}}\) and \(\underline{x}\) generate the ortho-symplectic Lie superalgebra and that
\begin{equation}\label{eq:Palgtrick}
    \ad_P(ab) = a\ad_P(b)
\end{equation}
whenever \(a\in\tama\) and \(\in\Cw\Cc\).

\begin{proposition}\label{prop:TamaDescriptionP}
    The TAMA is given as \(\tama = \ad_P(\Cent_{\Cw\Cc}(\mathfrak{spo}(2|1)_{\overline{0}}))\). 
\end{proposition}
\begin{proof}
    This was proven by Oste in his unpublished preprint. We sketch the proof here for convenience. Since \(\mathfrak{spo}(2|1)_{\overline{0}}\subset \mathfrak{spo}(2|1)\), it is clear that \(\Cent_{\Cw\Cc}(\mathfrak{spo}(2|1))\subset \Cent_{\Cw\Cc}(\mathfrak{spo}(2|1)_{\overline{0}})\). It is then a straightforward computation using the super-Jacobi identity to show that
    \[
    \ad_{\partial_{\underline{x}}}(P(a)) = 0 = \ad_{\underline{x}}(P(a))
    \]
    whenever \(a \in \Cent_{\Cw\Cc}(\mathfrak{spo}(2|1)_{\overline{0}})\), which proves the inclusion in the other direction.
\end{proof}

From the last proposition and from the fact that \(\mathfrak{spo}(2|1)_{\overline{0}}\cong \mathfrak{sl}(2)\) is entirely contained in the Weyl algebra, it follows that \(\Cent_{\Cw\Cc}(\mathfrak{spo}(2|1)_{\overline{0}}) = \Cent_{\Cw}(\mathfrak{spo}(2|1)_{\overline{0}})\otimes \Cc = \ama \otimes \Cc\). 
We can thus explicitly realize several elements of \(\tama\) by applying \(\ad_P\) to suitable elements of \(\Cw\Cc\). 
Following De Bie, van der Jeugt and Oste, for each basic element \(e^\gamma\), we obtain an element in \(\tama\) after applying \(-\tfrac{1}{2}\ad_P\). The normalizing factor of \(-\tfrac{1}{2}\) is there to ensure that, when \(1\leq i < j \leq n\), then
\begin{equation}\label{eq:2IndexSymmetries}
O_{ij} := -\frac{1}{2}\ad_P(e_ie_j) = L_{ij} + \frac{1}{2}e_ie_j
\end{equation} 
corresponds to a standard generator of the diagonal embedding of \(\mathfrak{so}(n,\bbc)\) in \(\Cw\Cc\). Following De Bie et al, the elements of \(\tama\) described in \eqref{eq:2IndexSymmetries} are called 2-index symmetries. More generally, a $k$-index symmetry, with \(2\leq k \leq n\) is given by
\begin{equation}\label{eq:kIndexSymmetries}
O_{a_1\cdots a_k} := -\frac{1}{2}\ad_P(e_{a_1}\cdots e_{a_k}) = \frac{k-1}{2} e_{a_1}\cdots e_{a_k} - \sum_{1\leq p < q \leq k}L_{a_pa_q}e_{a_p}e_{a_q}(e_{a_1}\cdots e_{a_k}),
\end{equation} 
with \(1\leq a_1<a_2<\cdots<a_k\leq n\). It will be convenient to write this ordered sequence of indices by $A=(a_1,a_2,\ldots,a_k)$ and write \[O_A = O_{a_1\cdots a_k}\] 
for the $k$-index symmetry of \eqref{eq:kIndexSymmetries}. Finally, note also that
\[
\ad_P(e_i) = e_i - \frac{1}{2} \ad_{\partial_{\underline x}}(\ad_{\underline x}(e_i)) = e_i - \ad_{\partial_{\underline x}}(x_i) = e_i-e_i = 0,
\] 
for all \(1\leq i \leq n\). The precise description of the \(k\)-index symmetries in \eqref{eq:kIndexSymmetries} imply the following structural result.

\begin{theorem}\label{thm:TamaGeneration}
    As an associative algebra, we have that \(\tama\) is generated by the 2-index symmetries \(\{O_{ij}\mid 1\leq i<j\leq n\}\) and the 3-index symmetries \(\{O_{ijk}\mid 1\leq i<j<k\leq n\}\).
\end{theorem}
\begin{proof}
    This was first proved by Oste in the setting of a rational Cherednik algebra. For convenience of the reader, we sketch the argument. Let \(\tama'\subseteq \tama\) denote the set of elements of \(\tama\) obtained as linear combinations of monomials involving finite products of 2- and 3-index symmetries. We will show that \(\tama\subseteq \tama'\).
    
    We have seen above that \(\tama = \ad_P(\ama\otimes \Cc)\). A general element in \(\ama\otimes \Cc\) is expressed as a linear combination of elements \(M\otimes e^\gamma\), where \(M\) is a finite product of the generators \(\{L_{ij}\mid 1\leq i<j\leq n\}\) of the \(\ama\). From \eqref{eq:2IndexSymmetries} we have \(L_{ij} = O_{ij}-\tfrac{1}{2}e_ie_j\), and since \(\ad_P(O_{ij}) = O_{ij}\), it follows that a generic element in \(\tama = \ad_P(\ama\otimes \Cc)\) is written as a linear combination of elements \(M_O\otimes \ad_P(e^\gamma)\), where \(M_O\) is a finite product of the 2-index symmetries \(\{O_{ij}\mid 1\leq i< j \leq n\}\). In other words, a general element of \(\tama\) can always be expressed as a linear combination of elements \(M_O\otimes O_A\), with $A=(a_1,a_2,\ldots,a_k)$. Hence, it suffices to show that any \(k\)-index symmetry, with \(k>3\), is contained in \(\tama'\). For that, by substituting \(L_{i_pi_q} = O_{i_pi_q}-\tfrac{1}{2}e_{i_p}e_{i_q}\) in \eqref{eq:kIndexSymmetries} and applying \(-\tfrac{1}{2}\ad_P\) to the resulting equation, and using \eqref{eq:Palgtrick}, we obtain, 
    \begin{equation}\label{eq:higherksymm}
    O_{A} = \frac{4}{k(k-3)}\sum_{1\leq p < q \leq k}(-1)^{p+q+1}O_{a_pa_q}O_{a_1\cdots \hat a_p\cdots \hat a_q \cdots a_k}
    \end{equation}
    for \(1\leq p < q \leq k\), \(k\geq 4\) and where \(\hat{a}\) means the corresponding index is removed. With a simple inductive argument, we conclude that any \(k\)-index symmetry is in \(\tama'\), and we are done.
\end{proof}

With the previous theorem in hand, it is useful to have formulas for the commutation between 2- and 3-index symmetries.

\begin{proposition}\label{prop:2and3commutation}
    Let \(\{i,j\}\) and \(\{p,q,r\}\) be sets of indices of cardinality 2 and 3, respectively. If \(|\{i,j\}\cap\{p,q,r\}|\) is even, then \(O_{ij}O_{pqr} = O_{pqr}O_{ij}\). Otherwise,
    \[
    O_{ij}O_{jqr} = O_{iqr}.
    \]
\end{proposition}

\begin{proof}
    Using \eqref{eq:Palgtrick}, note that
    \( [O_{ij},O_{pqr}] = -\tfrac{1}{2}\ad_P([O_{ij},e_{p}e_{q}e_{r}]) = -\tfrac{1}{4}\ad_P([e_ie_j,e_{p}e_{q}e_{r}])\). Assuming \(i,j,p,q,r\) distinct, it is then straightforward to compute
    \[[e_ie_j,e_{p}e_{q}e_{r}] = 0 = [e_ie_j,e_{i}e_{j}e_{r}] \quad \textup{and} \quad [e_ie_j,e_{j}e_{q}e_{r}] = 2e_{i}e_{q}e_{r},\]
    finishing the proof.
\end{proof}

Finally, since the 2-index symmetries correspond to standard generators of the diagonal embedding of \(\mathfrak{so}(n,\bbc)\), the operators  \(\{O_{ij} \mid 1 \leq i < j \leq n\}\) spans a subspace of \(\Cw\Cc\) isomorphic to the Lie algebra \(\mathfrak{so}(n, \mathbb{C})\) which correspond to the infinitesimal action of \(\mathrm{Pin}(n)\) on the polynomial-spinor space. Hence, it gives rise to a surjective homomorphism from \(\Cu(\mathfrak{so}(n, \mathbb{C}))\) onto a subalgebra of \(\tama\). In fact, we can be more precise about the image of this homomorphism.

\begin{proposition}\label{prop:SoImageIsEven}
    The image of the natural homomorphism \(\Cu(\mathfrak{so}(n, \mathbb{C}))\to\tama\) coincides with \(\tama_{\overline{0}}=\tama \cap \Cw\Cc_{\overline{0}}\), the even part of the TAMA.
\end{proposition}

\begin{proof}
    Let \(\Cu\) denote the image of the homomorphism \(\Cu(\mathfrak{so}(n, \mathbb{C}))\to\tama\). Since \(\Cu\) is generated by the 2-index symmetries, certainly \(\Cu \subseteq \tama_{\overline{0}}\). To prove the reverse inclusion, by Theorem \ref{thm:TamaGeneration} and Proposition \ref{prop:2and3commutation}, we need to guarantee that the product of any 3-index symmetries necessarily lands in \(\Cu\). We compute
    \begin{align*}
        O_{ijk}O_{pqr} &= -\frac{1}{2}\ad_P(O_{ijk}e_pe_qe_r)\\ 
        &= -\frac{1}{2}\ad_P((e_ie_je_k + L_{ij}e_k- L_{ik}e_j+ L_{jk}e_i)e_pe_qe_r)\\
        &= -\frac{1}{2}\ad_P\left(\left(\frac{1}{2}e_ie_je_k + O_{ij}e_k- O_{ik}e_j+ O_{jk}e_i\right)e_pe_qe_r\right)
    \end{align*} 
    which is a linear combination of \(k\)-index symmetries with \(k\) even and hence in \(\Cu\), by \eqref{eq:higherksymm}.
\end{proof}

The kernel of the surjective homomorphism \(\Cu(\mathfrak{so}(n, \mathbb{C}))\to\tama_{\bar 0}\) is the central object of study in this paper.

\subsection{Tableau relations}\label{s:TableauxRels}

As we saw in remark \ref{rem:CrossingRels}, in the AMA, the crossing relations \eqref{eq:CrossRels} are expressed, in the level of the associated graded algebra of the Weyl algebra, as
\[
L_{ij}L_{kl} + L_{ik}L_{lj} + L_{il}L_{jk} \equiv 0.
\]
 A more algebraic way to understand this relation is as follows. The natural filtration $\Cf^p(\Cw)$ of $\Cw$ induces a filtration on the AMA via $\ama^{(p)} = \Cf^{2p}(\Cw) \cap \ama$. Note that this coincides with defining $\ama^{(p)}$ as the span of all monomials of the type $\{L_D\mid D\in \mathscr{D}_p\}$ in light of Definition \ref{def:Diagrams} and \eqref{eq:DiagramMonAMA}.
Then, the symmetric group $S_4$ acts naturally on \(\ama^{(2)}\setminus \ama^{(1)}\) via permuting the four indices of any given monomial in that space. Now consider the standard Young tableau of shape $(1,1,1,1)$ and let $c_{(1,1,1,1)} \in \bbc S_4$ be the corresponding Young symmetriser. Given any sequence $(i,j,k,l)$ of four indices (not necessarily all entries distinct), note that
\[
c_{(1,1,1,1)}(L_{ij}L_{kl}) \equiv 8( L_{ij}L_{kl} + L_{ik}L_{lj}+L_{il}L_{jk})
\]
modulo terms in \(\ama^{(1)}\). Further, the stabiliser in \(S_4\) of the line \(\bbc(L_{ij}L_{kl}) \subset \ama^{(2)}\) is the subgroup \(H = \lpi (12),(34),(13)(24) \rpi\) and we can choose as a transversal set of the  coset space \(S_4/H\) the subgroup \(\Ct_2 := \lpi(234)\rpi\) so that the crossing-relation takes the form 
\begin{equation}\label{eq:AMASymmtriser}
\frac{1}{8}c_{(1,1,1,1)}(L_{ij}L_{kl}) = \sum_{g\in \Ct_2}g(L_{ij}L_{kl}) \equiv 0 
\end{equation}
modulo $\ama^{(1)}$.

In the TAMA, we can proceed in a similar fashion. But things are slightly more complicated since the ideal generating the kernel is not quadratic.

\begin{proposition}
    In the TAMA, the kernel of the natural homomorphism $\Cu(\mathfrak{so}(n,\bbc))\to \tama$ has no elements in filtration degree $3$ of $\Cu(\mathfrak{so}(n,\bbc))$
\end{proposition}

\begin{proof}

    Our proof is computational, and we only sketch the arguments. The key idea is to analyze the restriction to $\tama$ of the bilinear pairing $\beta$ from Section~\ref{sec:WeylCliff}, in which it suffices to compute things in the associated graded algebra.

    For each 4-tuple of indices \( i,j,k,l \) (not necessarily distinct), consider the set of quadratic elements
    \[
    \{O_{ij}O_{kl},\; O_{ik}O_{jl},\; O_{il}O_{jk}\}.
    \]
    Depending on equalities among indices, this set may contain fewer than three distinct elements, but we consider the span of the nonzero ones. If we examine the symmetric matrix \( B \) formed by pairing these elements via \( \beta \), a direct computation  shows that, in each case, the matrix \( B \) has full rank. Since this holds for all such 4-tuples, it follows that there are no (graded) degree 2 elements in the kernel of the homomorphism.

    A similar reasoning using the 15 (or, possibly less) monomials of (graded) degree 3 from any choice of a $6$-tuple of indices $i,j,k,l,m,n$ (not necessarily distinct) again only produces full rank matrices of pairings. 
\end{proof}

As a consequence of the previous proposition, we need to search for elements in the kernel that involve monomials on $\tama$ with at least $4$ generators. We remark that from \eqref{eq:higherksymm} the $4$-index symmetries satisfy, for any $4$-tuple $(i,j,k,l)$, 
\begin{equation}\label{eq:4-index}
    O_{ijkl} = O_{ij}O_{kl} + O_{ik}O_{lj} + O_{il}O_{jk}.
\end{equation}

The indices on the right-hand side of \eqref{eq:4-index} are exactly as in the crossing-relation for the AMA. We then consider products of $4$-index symmetries. The group $S_8$ naturally acts on the indices of a product $O_{ijkl}O_{pqrs}$ and we consider now the Young tableau of shape $(2,2,2,2)$ where we number the boxes from $1$ to $4$ in the first column and from $5$ to $8$ in the second column so that each column represents a $4$-index symmetry in a product. We then apply the Young symmetriser $c_{(2,2,2,2)}$ which is the product of an antisymmetriser on the $S_4\times S_4$ preserving the columns, followed by a symmetriser on the group 
\begin{equation} \label{eq:defofcT}
\Ct := \lpi (15),(26),(37),(48)\rpi \cong S_2^4
\end{equation}
that preserves the rows. We will now show that the Young symmetriser $c_{(2,2,2,2)}$, when applied to products of $4$-index symmetries, will produce non-trivial relations in AMA. Namely, we study the right hand side of the equation
\[ 
\frac{1}{(24!)^2}c_{(2,2,2,2)}(O_AO_B) = \sum_{g\in\Ct}\sigma(O_AO_B).
\]

\begin{lemma}\label{lem:CliffIds}
    Let $A = (a_1,a_2,\ldots,a_k)$ be an ordered sequence of distinct integers with an ordered pairs $\{ a_p,a_q\}$ such that $p <q\}$. Then,
    \begin{align*}
        e_{a_q}e_{a_p}e_{A\setminus\{a_p,a_q\}} &= (-1)^{p+q}e_A \\
        e_{A\setminus\{a_p,a_q\}}e_A &= (-1)^{k(k-1)/2}(-1)^{p+q}e_{a_p}e_{a_q}\\
        e_{a_p}e_{a_q}e_A &= (-1)^{p+q}e_{A\setminus\{a_p,a_q\}}
    \end{align*}
\end{lemma}
\begin{proof}
    For the first equation, note that $e_{a_p}e_{A\setminus\{a_p,a_p\}} = (-1)^{p-1} e_{A\setminus\{a_q\}}$, since we need to swap $p-1$ indices to arrange the sequence with $a_p$ omitted. Likewise, we swap $q-1$ indices to rearrange $e_{a_q}e_{A\setminus\{a_q\}}$, which settles this equation.
    For the second, we have $e_{A\setminus\{a_p,a_p\}}e_A = \alpha e_{a_p}e_{a_q}$. By left multiplication, the constant $\alpha$ satisfies
    \[
    \alpha = e_{a_q}e_{a_p}e_{A\setminus\{a_p,a_p\}}e_A = (-1)^{p+q}e_A^2 = (-1)^{k(k-1)/2}(-1)^{p+q},
    \]
    where the first equation was used. For the last equation, we know that $e_{a_p}e_{a_q}e_A = \alpha e_{A\setminus\{a_p,a_p\}}$, from which, using again the first equation, we get
    \[ 
    e_A^2 = \alpha(e_Ae_{a_q}e_{a_p}e_{A\setminus\{a_p,a_p\}}) = \alpha(-1)^{p+q}e_A^2
    \]
    which implies $\alpha=(-1)^{p+q}$.
\end{proof}

\begin{theorem}\label{t:relsinideal}
 Let $A = (a_1,a_2,a_3,a_4)$ and $B = (b_1,b_2,b_3,b_4)$ be sequences of four distinct numbers. Then,
 \[ \sum_{ \sigma \in \mathcal{T}} \sigma (O_AO_B) = 12 \delta_{A,B}-  \sum_{g \in \mathcal{T}} \sigma \left( \sum_{(a,b) \subset  A \cap B} O_{A \setminus ( a,b)} O_{B \setminus ( a,b)} \right).\] 
 \end{theorem}

 \begin{proof}
From \eqref{eq:kIndexSymmetries} and using $L_{ij} = O_{ij} - \frac{1}{2}e_ie_j$, we obtain
\begin{align*}
O_Ae_B &= -\frac{3}{2}e_A -\sum_{p<q} (-1)^{p+q}O_{a_pa_q}e_{A\setminus\{a_p,a_q\}}\\
&= -\frac{3}{2}e_A -\sum_{p<q} (-1)^{p+q}O_{A\setminus\{a_p,a_q\}}e_{a_pa_q},
\end{align*}
where we used in the last equation that the sum of any two indices has the same parity as the sum of their complement in a sequence with $4$ terms. Hence, 
\begin{equation}\label{eq:Geneightterms}
O_Ae_B = -\frac{3}{2}e_Ae_B -\sum_{p<q} (-1)^{p+q}O_{A\setminus\{a_p,a_q\}}e_{a_pa_q}e_B.
\end{equation}

 Note that the first term in the right side of (\ref{eq:Geneightterms}) will be antisymmetric in the terms $A \cup B \setminus A \cap B$, thus if $A \neq B$ then this term will go to zero under symmetrisation by $\mathcal{T}$. If $A = B$ then this term is $-\frac{3}{2}e_{A}^2 = -\frac{3}{2}$. Focusing on the Clifford terms of the form $e_{a_pa_q}e_B$ with $(a_p,a_q) \subset A$, if $(a_p,a_q)$ is not a subsequence of $B$ then this term will be an antisymmetric Clifford element of degree $4$ or higher. Thus, it will be antisymmetric in at least one of the elements of $\mathcal{T}$. Hence symmetrising by $\mathcal{T}$ will annhilate such elements. The only contributing elements on the right hand side of (\ref{eq:Geneightterms}) are monomials $O_{A\setminus\{a_p,a_q\}}e_{a_pa_q}e_B$ such that $(a_p,a_q) \subset B$. Using Lemma \ref{lem:CliffIds} and (\ref{eq:Geneightterms}) we obtain
 \[
    \sum_{\sigma \in \mathcal{T}}\sigma(O_Ae_B)=-24\delta_{A,B}-  \sum_{\sigma \in \mathcal{T}} \sum_{(a_p,a_q) \subset  A \cap B} \sigma (O_{A \setminus ( a_p,a_q)} e_{B \setminus ( a_p,a_q)}) .
\]
Applying the extremal projector $-\frac{1}{2}\ad_P$ finalizes the prove of the theorem. 
 \end{proof}

 \begin{example}\label{eg:rel4and5}
 When $A=\{a,b,c,d\} =\{ 1,2,3,4\} = \{w,x,y,z\} = B$, every element $\sigma \in \mathcal{T}$ is such that $\sigma(O_A O_B) = O_A O_B$.  Also, the element $ \sum_{(a,b) \in \{1,2,3,4\}} O_{A \setminus (a,b)}^2$ is stable under every $\sigma \in \mathcal{T}$. Thus the equation 
    \[
    \sum_{\sigma \in \mathcal{T}} \sigma(O_AO_B) =12 \delta_{A,B} -\sum_{\sigma \in \mathcal{T}} (O_{12}^2 + O_{13}^2 + O_{14}^2 + O_{23}^2 + O_{24}^2 + O_{34}^2) \] becomes, once divided by $16 = |\mathcal{T}|$,
    \[ O_{1234}^2 = \frac{3}{4} - O_{12}^2 - O_{13}^2 - O_{14}^2 - O_{23}^2 - O_{24}^2 - O_{34}^2.\]
 When $A = (i,j,k,l)$ with $B = (i,j,k,s)$, after dividing by $8$, the equation of Theorem \ref{t:relsinideal}  yields the equation
\[O_AO_B  + O_B O_A=  -\sum_{a \in (i,j,k)} O_{al}O_{as} + O_{as}O_{al}.\]
\end{example}

\subsection{Main conjecture}
We can now state our main conjecture.
Consider $X_{ij}:= E_{ij} - E_{ji} \in U(\mathfrak{so}(n))$ and for $A= (a_1,a_2,a_3,a_4)$ a sequence of four distinct indices define 
\[ X_A = X_{a_1a_2} X_{a_3a_4} + X_{a_1a_3} X_{a_4a_2} + X_{a_1a_4} X_{a_2a_3}. \]
     
\begin{conjecture}\label{con:maincon}
    
       Let $I$ be the ideal in the exact sequence 
      \[0 \to I \to U(\fso(n)) \to \tama_{\overline{0}} \to 0.\]
      Then, $I$ is generated by elements of the form
 \begin{equation}\label{e:relsinideal} \sum_{ \sigma \in \mathcal{T}} \sigma \left( X_AX_B  +   \sum_{(a,b) \subset  A \cap B} X_{A \setminus ( a,b)} X_{B \setminus ( a,b)} \right) - 12 \delta_{A,B},\end{equation}
 with $A = (a_1,a_2,a_3,a_4)$, $B = (b_1,b_2,b_3,b_4)$ sequences of four distinct numbers, and $\mathcal{T}$ defined as in (\ref{eq:defofcT}).
   In particular, there exists a presentation of the even subalgebra of the TAMA 
   \[ \tama_{\overline{0}} \cong  U(\mathfrak{so}(n)) \big/ \left\langle  \sum_{ \sigma \in \mathcal{T}} \sigma \left( X_AX_B  +   \sum_{(a,b) \subset  A \cap B} X_{A \setminus ( a,b)} X_{B \setminus ( a,b)} \right) - 12 \delta_{A,B} \right\rangle.
   \]
\end{conjecture}
Note that Theorem \ref{t:relsinideal} shows that the relations (\ref{e:relsinideal}) are all contained in the ideal $I$, the conjecture is predominately concerned with the statement that $I$ is generated by these relations. In the section below we prove Conjecture \ref{con:maincon} when $n=4$ and $n=5$. The technique that we use to show that $I$ is generated by relations (\ref{e:relsinideal}) is the following. First by applying relations of  form \eqref{e:relsinideal} we find a spanning set for all monomials. We refer to the elements in this spanning set as uncrossable monomials. We then show that the set of uncrossable monomials is linearly independent in $\mathcal{WC}$ and therefore in \(\tama_{\overline{0}}\).

\section{Uncrossable diagrams and linear independence} \label{sec:lastsection}
In this section we prove Conjecture \ref{con:maincon} for $n=4,5$. Our strategy is as follows: we wish to show that the ideal $I$ is generated by relations (\ref{e:relsinideal}). We first (Section \ref{sec:uncrossable}) find a set of monomials, which we call uncrossable monomials and show that this set is linearly independent in $\mathcal{WC}$, in fact we work entirely in $\mathsf{gr}(\mathcal{WC})$, since linear independence in the associated graded implies linear dependence. The results in this section hold for all $n$. In Section \ref{sec:prooffor4and5}, specifying to $n=4$ and $5$, we apply relations of the  form \eqref{e:relsinideal} and show show that the set of uncrossable monomials is a spanning set for subalgebra of $\tama$ generated by $O_{ij}$. Thus proving that when $n \in \{ 4,5\}$ the ideal $I$ is generated by relations (\ref{e:relsinideal}).   Throughout this section we work modulo lower order terms, or more accurately in the associated graded algebras $\mathsf{gr} (\mathcal{W} \otimes \clif) \cong S(V \oplus V^*) \otimes \bigwedge (V)$. We can do this since proving linear independence can be done in the associated graded, and using induction filtered piece by filtered piece we can also prove spanning results in the associated graded. Since both the TAMA and the AMA-Clifford algebra inherit their filtration from the Weyl Clifford algebras there are linear maps $\mathsf{gr}(\tama) \to \mathsf{gr}(\ama \otimes  \clif) \to \mathsf{gr}(\mathcal{W}\clif)$. Throughout this section an element written in sans serif will be in the associated graded algebra, for example for $O_{ij} \in \tama$, then $\mathsf{O}_{ij} \in \mathsf{gr}(\tama)$, for $L_{ij} \in \ama$, then $\mathsf{L}_{ij} \in \mathsf{gr}(\ama)$ and for $e_{k} \in \Cc$, then $\mathsf{e}_{k} \in \mathsf{gr}(\Cc)$.

\subsection{Uncrossable monomials and linear independence} \label{sec:uncrossable}

\begin{definition}[TAMA Diagram] \label{def:TAMAdiagram} We define a diagram $D$ associated to a monomial of the form $\prod \mathsf{O}_{ij}^{m_{ij}}$ to be the graph on $n$ vertices with chord $(i,j)$ labeled by $m_{ij}$.  
\end{definition}

Recall that a diagram was crossed if there existed $i < j < k <l$ such that $m_{ik} >1$ and $m_{jl} >1$ and uncrossed if this never happens. We say a diagram $D$ is uncrossable if either $D$ is uncrossed or $D = D' \cup (i,j)$ where $D'$ is uncrossed. That is, there is a single chord that one can remove from $D$ to form an uncrossed diagram. We say that a monomial is uncrossable if its associated diagram is uncrossable.

\begin{figure}
    \centering
 \begin{tikzpicture}[
    baseline=(current bounding box.center),
    scale=1.2,
    every node/.style={circle, fill=black, inner sep=1pt},
    midarrow/.style={
      blue,
      thick,
      postaction={decorate},
      decoration={
        markings,
        mark=at position 0.5 with {\arrow{>}}
      }
    }
  ]

    \foreach \i in {1,...,4} {
      \node (v\i) at ({360/4 * (\i - 1)}:1) {};
      \node[draw=none, fill=none] at ({360/4 * (\i - 1)}:1.2) {\(v_{\i}\)};
    }

    \draw[midarrow] (v1) -- (v2) node[draw=none,fill=none,midway,above right] {$w$};
    \draw[midarrow] (v2) -- (v3) node[draw=none,fill=none,midway,above left] {$x$};
    \draw[midarrow] (v3) -- (v4) node[draw=none,fill=none,midway,below left] {$y$};
    \draw[midarrow] (v1) -- (v4)node[draw=none,fill=none,midway,below right] {$z$};
    \draw[midarrow40] (v1) -- (v3) node[draw=none,fill=none,near start,above right] {$a$};
    \draw[midarrow60] (v2) -- (v4)node[draw=none,fill=none,near start,below left] {$b$};
    \end{tikzpicture}
    \caption{The TAMA diagram associated to $\mathsf{O}_{12}^{w}\mathsf{O}_{23}^{x}\mathsf{O}_{34}^{y}\mathsf{O}_{14}^{z}\mathsf{O}_{13}^{a}\mathsf{O}_{24}^{b}$ }
    \label{fig:TAMAD}
\end{figure}
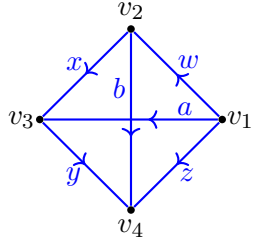

\begin{definition}[AMA-exterior Diagram]\label{d:AMAEdiags}
     We define a diagram $\underline{D}$ associated to a monomial of the form $\prod \mathsf{L}_{ij}^{n_{ij}}\mathsf{e}_k^{\gamma_k}$ to be the graph on $n$ vertices with chord $(i,j)$ labelled by $n_{ij}$ and vertex $k$ coloured black if $\gamma_k =1$ and white if $\gamma_k =0$.   
\end{definition}

\begin{figure}
    \centering
     \begin{tikzpicture}[
    baseline=(current bounding box.center),
    scale=1.2,
    every node/.style={circle, fill=black, inner sep=1pt},
    midarrow/.style={
      blue,
      thick,
      postaction={decorate},
      decoration={
        markings,
        mark=at position 0.5 with {\arrow{>}}
      }
    }
  ]

    \foreach \i in {1,...,4} {
      \node[draw=none,fill=none] (v\i) at ({360/4 * (\i - 1)}:1) {};
      \node[draw=none, fill=none] at ({360/4 * (\i - 1)}:1.4) {\(v_{\i}\)};
    }

    \draw[midarrow] (v1) -- (v2) node[draw=none,fill=none,midway,above right] {$w$};
    \draw[midarrow] (v2) -- (v3) node[draw=none,fill=none,midway,above left] {$x$};
    \draw[midarrow] (v3) -- (v4) node[draw=none,fill=none,midway,below left] {$y$};
    \draw[midarrow] (v1) -- (v4)node[draw=none,fill=none,midway,below right] {$z$};
    \draw[midarrow40] (v1) -- (v3) node[draw=none,fill=none,near start,above right] {$a$};
    \draw[midarrow60] (v2) -- (v4)node[draw=none,fill=none,near start,below left] {$b$};
    
    \filldraw[color=black, fill=white] ({360/4 * (2 - 1)}:1) circle (3pt);
     \filldraw[color=black, fill=white] ({360/4 * (3 - 1)}:1) circle (3pt);
      \filldraw[color=black, fill=black] ({360/4 * (1 - 1)}:1) circle (3pt);
       \filldraw[color=black, fill=black] ({360/4 * (4 - 1)}:1) circle (3pt);

    \end{tikzpicture}
    
\caption{The AMA-exterior diagram associated to  $\mathsf{O}_{12}^w\mathsf{O}_{23}^x\mathsf{O}_{34}^y\mathsf{O}_{14}^z\mathsf{O}_{13}^a\mathsf{O}_{24}^b\mathsf{e}_1 \mathsf{e}_4$ 
}
    
    \end{figure}
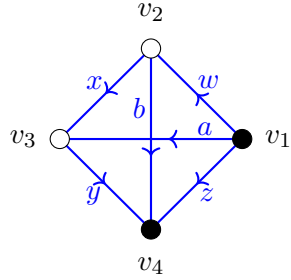

Feigin and Hakobyan \cite{FH15} prove that the monomials corresponding to uncrossed diagrams form a basis for the AMA. A basis for the AMA-exterior algebra is given by the outer product of a basis for AMA and a basis for the exterior algebra. A basis for the exterior algebra is $\{\mathsf{e}_A = \bigwedge_{a \in A} \mathsf{e}_a: A \subset \{ 1,\ldots,n\}\}$ corresponds to the $2^n$ different colouring of $n$ vertices. Thus we can also say that the monomials $\prod \mathsf{L}_{ij}\mathsf{e}_k^{\gamma_k}$ associated to uncrossed AMA-exterior diagrams form a basis for the  associated graded of AMA-Clifford algebra. We represent this basis by ${0,1}$ labeled uncrossed diagrams (AMA-exterior diagrams in Definition \ref{d:AMAEdiags}).

Since $\mathsf{O}_{ij} = \mathsf{L}_{ij} + \frac{1}{2} \mathsf{e}_i\mathsf{e}_j$, we can express every monomial in $\mathsf{O}_{ij}$ uniquely as a sum of monomials in $\mathsf{L}_{ij},\mathsf{e}_k$ which correspond to uncrossed AMA-exterior diagrams. A TAMA diagrams corresponds to a monomials in $\mathsf{O}_{ij}$, thus we can uniquely express every TAMA diagram as a sum of uncrossed AMA-exterior diagrams. We say that the AMA-exterior support of a TAMA diagram $D$ is the uncrossed AMA-exterior diagrams that have non zero coefficient in the expression of $D$ as a linear span of uncrossed AMA-exterior diagrams. This process can be done procedurally by taking the monomial $\prod \mathsf{O}_{ij}^{m_{ij}}$, expressing each individual $\mathsf{O}_{ij}$ as $\mathsf{L}_{ij} + \frac{1}{2}\mathsf{e}_i\mathsf{e}_j$ expanding in the algebra $\mathsf{gr}( \ama \otimes \clif)$ and then using the crossing relations for the AMA to rewrite any crossed monomials as uncrossed monomials. Alternatively, one can take the monomial $\prod \mathsf{O}_{ij}$ corresponding to $D$ and invoke the linear map $\mathsf{gr} (\tama) \to \mathsf{gr} (\ama \otimes\clif)$ which expresses a TAMA diagram in terms of uncrossed AMA-exterior diagrams. This process is highlighted in Figure \ref{fig:TAMA2AMAE}.

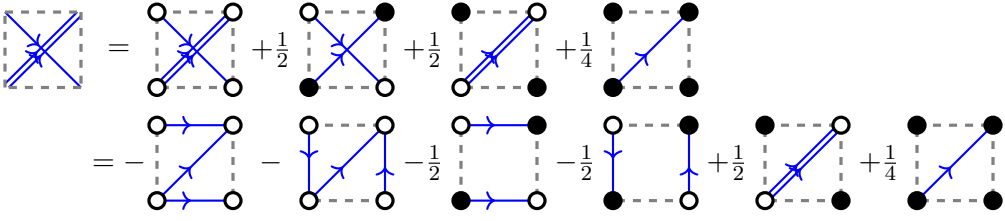
\begin{figure}
    \centering
    \begin{tikzpicture}
      \begin{scope}[very thick,decoration={
    markings,
    mark=at position 0.45 with {\arrow{>}}}
    ]

\draw[blue, thick, postaction={decorate}] (0,0.05) -- (.95,1);
\draw[blue, thick,postaction={decorate}] (0.05,0) -- (1,.95);
\draw[blue, thick,postaction={decorate}] (0,1) -- (1,0);

 \draw[dashed, gray] (0,0) -- (0,1);
 \draw[dashed, gray] (0,1) -- (1,1);
 \draw[dashed, gray] (1,1) -- (1,0);
 \draw[dashed, gray] (1,0) -- (0,0);

\draw(1.5,.5) node{$=$};

\draw[blue, thick, postaction={decorate}] (2,0.05) -- (2.95,1);
\draw[blue, thick, postaction={decorate}] (2.05,0) -- (3,.95);
\draw[blue, thick, postaction={decorate}] (2,1) -- (3,0);

\draw[dashed, gray] (2,0) -- (2,1);
 \draw[dashed, gray] (2,1) -- (3,1);
 \draw[dashed, gray] (3,1) -- (3,0);
 \draw[dashed, gray] (3,0) -- (2,0);

\filldraw[color=black, fill=white] (2,0) circle (3pt);
\filldraw[color=black, fill=white] (2,1)  circle (3pt) ;
\filldraw[color=black, fill=white] (3,1) circle (3pt)  ;
\filldraw[color=black, fill=white] (3,0)  circle (3pt) ;

\draw(3.5,.5) node{$+\frac{1}{2}$};

\draw[blue, thick, postaction={decorate}] (4,0) -- (5,1);
\draw[blue, thick, postaction={decorate}] (4,1) -- (5,0);
\draw[dashed, gray] (4,0) -- (4,1);
 \draw[dashed, gray] (4,1) -- (5,1);
 \draw[dashed, gray] (5,1) -- (5,0);
 \draw[dashed, gray] (5,0) -- (4,0);
\filldraw[color=black, fill=black] (4,0) circle (3pt);
\filldraw[color=black, fill=white] (4,1)  circle (3pt) ;
\filldraw[color=black, fill=black] (5,1) circle (3pt)  ;
\filldraw[color=black, fill=white] (5,0)  circle (3pt) ;

\draw(5.5,.5) node{$+\frac{1}{2}$};

\draw[blue, thick, postaction={decorate}] (6,0.05) -- (6.95,1);
\draw[blue, thick, postaction={decorate}] (6.05,0) -- (7,.95);
\draw[dashed, gray] (6,0) -- (6,1);
 \draw[dashed, gray] (6,1) -- (7,1);
 \draw[dashed, gray] (7,1) -- (7,0);
 \draw[dashed, gray] (7,0) -- (6,0);

\filldraw[color=black, fill=white] (6,0) circle (3pt);
\filldraw[color=black, fill=black] (6,1)  circle (3pt) ;
\filldraw[color=black, fill=white] (7,1) circle (3pt)  ;
\filldraw[color=black, fill=black] (7,0)  circle (3pt) ;

\draw(7.5,.5) node{$+\frac{1}{4}$};

\draw[blue, thick, postaction={decorate}] (8,0) -- (9,1);

\draw[dashed, gray] (8,0) -- (8,1);
 \draw[dashed, gray] (8,1) -- (9,1);
 \draw[dashed, gray] (9,1) -- (9,0);
 \draw[dashed, gray] (9,0) -- (8,0);
\filldraw[color=black, fill=black] (8,0) circle (3pt);
\filldraw[color=black, fill=black] (8,1)  circle (3pt) ;
\filldraw[color=black, fill=black](9,1) circle (3pt)  ;
\filldraw[color=black, fill=black] (9,0)  circle (3pt) ;

\draw(1.5,-1) node{$=-$};

 \draw[blue, thick, postaction={decorate}] (2,-1.5) -- (3,-1.5);
 \draw[blue, thick, postaction={decorate}] (2,-.5) -- (3,-.5);
\draw[blue, thick, postaction={decorate}] (2,-1.5) -- (3,-.5);

\draw[dashed, gray] (2,-1.5) -- (2,-.5);
 \draw[dashed, gray] (3,-.5) -- (3,-1.5);

\filldraw[color=black, fill=white] (2,-1.5) circle (3pt);
\filldraw[color=black, fill=white] (2,-.5)  circle (3pt) ;
\filldraw[color=black, fill=white] (3,-.5) circle (3pt)  ;
\filldraw[color=black, fill=white] (3,-1.5)  circle (3pt) ;

\draw(3.5,-1) node{$-$};

 \draw[blue, thick, postaction={decorate}] (5,-1.5) -- (5,-.5) ;

 \draw[blue, thick, postaction={decorate}] (4,-.5) -- (4,-1.5);

\draw[blue, thick, postaction={decorate}] (4,-1.5) -- (5,-.5);
 \draw[dashed, gray] (4,-.5) -- (5,-.5);
 \draw[dashed, gray] (5,-1.5) -- (4,-1.5);
\filldraw[color=black, fill=white] (4,-1.5) circle (3pt);
\filldraw[color=black, fill=white] (4,-.5)  circle (3pt) ;
\filldraw[color=black, fill=white] (5,-.5) circle (3pt)  ;
\filldraw[color=black, fill=white] (5,-1.5)  circle (3pt) ;

\draw(5.5,-1) node{$-\frac{1}{2}$};

 \draw[blue, thick, postaction={decorate}] (6,-1.5) -- (7,-1.5) ;

 \draw[blue, thick, postaction={decorate}] (6,-.5) -- (7,-.5);

 \draw[dashed, gray] (6,-1.5) -- (6,-.5);
 \draw[dashed, gray] (7,-.5) -- (7,-1.5);

\filldraw[color=black, fill=black] (6,-1.5) circle (3pt);
\filldraw[color=black, fill=white] (6,-.5)  circle (3pt) ;
\filldraw[color=black, fill=black] (7,-.5) circle (3pt)  ;
\filldraw[color=black, fill=white] (7,-1.5)  circle (3pt) ;

\draw(7.5,-1) node{$-\frac{1}{2}$};

 \draw[blue, thick, postaction={decorate}] (9,-1.5) -- (9,-.5) ;

 \draw[blue, thick, postaction={decorate}] (8,-.5) -- (8,-1.5);

 \draw[dashed, gray] (8,-.5) -- (9,-.5);
 \draw[dashed, gray] (9,-1.5) -- (8,-1.5);

\filldraw[color=black, fill=black] (8,-1.5) circle (3pt);
\filldraw[color=black, fill=white] (8,-.5)  circle (3pt) ;
\filldraw[color=black, fill=black] (9,-.5) circle (3pt)  ;
\filldraw[color=black, fill=white] (9,-1.5)  circle (3pt) ;

\draw(9.5,-1) node{$+\frac{1}{2}$};

\draw[blue, thick, postaction={decorate}] (10,-1.45) -- (10.95,-.5);
\draw[blue, thick, postaction={decorate}] (10.05,-1.5) -- (11,-.55);
 \draw[dashed, gray] (10,-1.5) -- (10,-.5);
 \draw[dashed, gray] (10,-.5) -- (11,-.5);
 \draw[dashed, gray] (11,-.5) -- (11,-1.5);
 \draw[dashed, gray] (11,-1.5) -- (10,-1.5);
\filldraw[color=black, fill=white] (10,-1.5) circle (3pt);
\filldraw[color=black, fill=black] (10,-.5)  circle (3pt) ;
\filldraw[color=black, fill=white] (11,-.5) circle (3pt)  ;
\filldraw[color=black, fill=black] (11,-1.5)  circle (3pt) ;

\draw(11.5,.-1) node{$+\frac{1}{4}$};

\draw[blue, thick, postaction={decorate}] (12,-1.5) -- (13,-.5);
 \draw[dashed, gray] (12,-1.5) -- (12,-.5);
 \draw[dashed, gray] (12,-.5) -- (13,-.5);
 \draw[dashed, gray] (13,-.5) -- (13,-1.5);
 \draw[dashed, gray] (13,-1.5) -- (12,-1.5);
\filldraw[color=black, fill=black] (12,-1.5) circle (3pt);
\filldraw[color=black, fill=black] (12,-.5)  circle (3pt) ;
\filldraw[color=black, fill=black](13,-.5) circle (3pt)  ;
\filldraw[color=black, fill=black] (13,-1.5)  circle (3pt) ;
\end{scope}
\end{tikzpicture}
    \caption{Expressing a TAMA diagram associated to $\mathsf{O}_{13}^2\mathsf{O}_{24}$ in terms of uncrossed AMA-exterior diagrams}
    \label{fig:TAMA2AMAE}
\end{figure}

\begin{lemma}\label{l:collectionoflinindep}
    Let $\mathcal{D}$ be a collection of TAMA diagrams. Suppose that each diagram $D \in \mathcal{D}$  has a unique uncrossed AMA-exterior diagram in the support of $D$ when expressed as a linear combination of uncrossed AMA-exterior diagrams. Then $\mathcal{D}$ is a linear independent set.   
\end{lemma}

\begin{proof}
    For every $D$ in $\mathcal{D}$ denote $\widehat{\underline{D}}$ the unique AMA-exterior diagram in the support of $D$. Suppose there is a linear dependence on $\mathcal{C}$, 
    \[ \sum_{D \in \mathcal{D}} \lambda_D D.\]
    Then expanding this linear dependence in terms of uncrossed AMA-exterior diagrams gives a linear dependence on uncrossed AMA-exterior diagrams. Since  $\widehat{\underline{D}}$ is only in the support of $D$ and no other diagram in $\mathcal{D}$, then the coefficient  for $\widehat{\underline{D}}$ in this linear dependence is $\lambda_D$.  Then $\lambda_D=0$ since all uncrossed AMA-exterior diagrams are linearly independent. 
\end{proof}

\begin{lemma}\label{l:uncrossablearelinind}
    The uncrossable diagrams are linearly independent.
\end{lemma}

\begin{proof}
    By definition an uncrossable diagram $D$ is such that $D = D' \cup (i,j)$ with $D'$ an uncrossed diagram. Consider the uncrossed AMA-exterior diagram $\underline{D}'$ that has the same chords as $D'$ with the vertices $i$ and $j$ coloured black. Then $\underline{D}'$ is in the support of $D$ when expressed as a linear combination of uncrossed AMA-exterior diagrams. Furthermore, the only TAMA diagram that has $\underline{D}'$ in its support is $D$, since there is a unique way to join the coloured vertex $i$ to the coloured vertex $j$ with a chord.
    Therefore, the set of uncrossable diagrams satisfies the condition of Lemma \ref{l:collectionoflinindep} and hence must be linearly independent.
\end{proof}

\subsection{Spanning sets and a proof of Conjecture \ref{con:maincon} when \texorpdfstring{$n =4$}{n = 4} and \texorpdfstring{$n =5$}{n = 5}} \label{sec:prooffor4and5}

We first describe the monomials  $\mathsf{O}_{12}^{m_{12}}\mathsf{O}_{23}^{m_{23}}\mathsf{O}_{34}^{m_{34}}\mathsf{O}_{14}^{m_{14}}\mathsf{O}_{13}^{m_{13}}\mathsf{O}_{24}^{m_{24}}$ associated to uncrossable diagrams when $n=4$ and $n = 5$. Studying figure \ref{fig:TAMAD} (where we used $a$ for $m_{13}$, $b$ for $m_{24}$ and so on), we can see that, for $n=4$ the uncrossed diagrams are precisely the diagrams which have either $m_{13}$ or $m_{24}$ zero. Equivalently in the notation of Figure \ref{fig:TAMAD}, either $a$ or $b$ is zero.  Therefore, the uncrossable diagrams, those that we can create by adding a chord to an uncrossed diagrams are diagrams such that either $m_{13} \leq 1$ or $m_{24} \leq 1\}$ thus 
\[ \{ \prod \mathsf{O}_{ij}^{m_{ij}} \mid m_{13} \leq 1 \text{ or } m_{24} \leq 1 \}\]
is the set of monomials associated to uncrossable diagrams.

\begin{lemma}When $n=5$ the set of uncrossable diagrams are quantified by the conditions
\begin{align}\label{eq:uncross5cond}
    |\{ m_{i,i+2} > 0 \}| \leq 3& \\ \nonumber
      m_{i,i+2} + m_{i,i+3} \leq 1  &  \text{ if } m_{i,i+2} \geq2
\end{align}
\end{lemma}
\begin{proof}
Suppose that a diagram does not satisfy conditions (\ref{eq:uncross5cond}) then either there are $4$ different internal chords or   $m_{i,i+1} \geq 2$ and $m_{i,i+2}+m_{i,i+3} \geq 2$.
If there are $4$ or more different internal chords, it is impossible to remove a chord to leave an uncrossed diagram. If  $m_{i,i+1} \geq 2$ and $m_{i,i+2}+m_{i,i+3} \geq 2$ then one needs to remove more than one chord to uncross this diagram. The contrapositive now states that any uncrossable diagram must satisfy (\ref{eq:uncross5cond}). Furthermore, if a diagram satisfies conditions (\ref{eq:uncross5cond}) then there are at most three internal chords. Without loss of generality let these chords be $m_{i,i+2}, m_{i,i+3},m_{i-1,i+1}$ forming an 'A' pointing at $i$. if $m_{i-1,i+1} =1$ then one can remove this chord to create an uncrossed diagram. If $m_{i-1,i+1} \geq 2$ then by the second condition of (\ref{eq:uncross5cond}), $m_{i,i+1},m_{i,i+3} \leq 1$, thus one can remove this single chord to form an uncrossed diagram. 
\end{proof}
Thus 
\[ \{ \prod \mathsf{O}_{ij}^{m_{ij}} \mid|\{ m_{i,i+2} > 0 \}| \leq 3 \text{ and }
        m_{i,i+2} + m_{i,i+3} \leq 1  \text{ if } m_{i,i+2} \geq2\}\] 
is the set of monomials associated to uncrossable diagrams.

\begin{theorem}
    When $n=4,5$ the uncrossable diagrams correspond to a  basis of monomials for the subalgebra of $ \mathcal{W} \otimes \clif$ generated by $O_{ij}$.
\end{theorem}

\begin{proof}[Proof when $n=4$] Since Lemma \ref{l:uncrossablearelinind} shows that uncrossable diagrams correspond to linear independent monomials, we are left to prove that the uncrossable diagrams correspond to a spanning set of the subalgebra generated by $O_{ij}$. We again work in the associated graded. By applying induction and proving that the uncrossable diagrams of degree $k$ span the associated graded elements of degree $k$ then proving that the associated graded uncrossable monomials span implies that their counterparts in the filtered algebra also span.
     If $n=4$ there is a single relation as calculated in Theorem \ref{t:relsinideal}. This corresponds to the following relation in $\mathsf{gr} (\mathcal{W} \clif)$
\begin{equation}\label{eq:doublecrossingrel}
    \mathsf{O}_{1234}^2  = (\mathsf{O}_{12}\mathsf{O}_{34} - \mathsf{O}_{13}\mathsf{O}_{42} + \mathsf{O}_{14}\mathsf{O}_{32})^2=0 
\end{equation} 

We refer to the single relation (\ref{eq:doublecrossingrel}) as the double crossing relation. We use this relation to rewrite any monomial including $\mathsf{O}_{13}^{m_{13}}\mathsf{O}_{24}^{m_{23}}$ for $m_{13}$ and $m_{23}$ greater than one in terms of monomials in the set 
\[ \{ \prod \mathsf{O}_{ij}^{m_{ij}} \mid m_{13} \leq 1 \text{ or } m_{24} \leq 1 \}.\]
This process must terminate since a diagram is uncrossable if and only if it contains no double crosses ($m_{13}$ and $m_{24} \geq 2$). Hence if we start with a monomial correspond to a diagram with $k$ double crosses then we can apply relation (\ref{eq:doublecrossingrel}) to write this monomial as a span of monomials associated to diagrams with at most $k-1$ double crosses. Successively applying this process shows that every monomial in $\mathsf{O}_{ij}$ can be written as a sum of monomials associated to diagrams with no double cross. For $n=4$ every diagram not including a double cross is uncrossable. Hence the set of monomials corresponding to uncrossable diagram span the subalgebra generated by $\mathsf{O}_{ij}$.
\end{proof}

\begin{proof}[Proof when $n=5$]
When $n=5$ there are $15$ relations coming from Theorem \ref{t:relsinideal}, $5$ of which use $4$ of the five vertices and are of the same form as (\ref{eq:doublecrossingrel}). These $5$ are equivalent to picking $A = B \subset \{1,\ldots,5\}$. There are ten more relations that come from the ten choices of $A, B$ subsets of size $4$ such that $|A \cap B| = 3$. We split these ten relations into two types.  The first type when $A\cap B$ is a set of $3$ (modulo $5$) consecutive numbers. The second five when $A \cap B$ is not consecutive. More precisely $A \cap B $ is a rotation of $\{ 1,3,4\}$. Recall 
\[ \mathsf{O}_{abcd} = \mathsf{O}_{ab}\mathsf{O}_{cd} + \mathsf{O}_{ac}\mathsf{O}_{bd} + \mathsf{O}_{ad}\mathsf{O}_{cb},\]

Thus the equation of the form 
\[ \mathsf{O}_{1234}\mathsf{O}_{2345} + \mathsf{O}_{2345}\mathsf{O}_{1234} = 0\] which, since $\mathsf{O}_A$ and $\mathsf{O}_B$ commute in the associated graded is equivalent to 
\[ \mathsf{O}_{1234}\mathsf{O}_{2345} = 0,\] then becomes,

\[ (\mathsf{O}_{12}\mathsf{O}_{34} +\mathsf{O}_{13}\mathsf{O}_{24} + \mathsf{O}_{14}\mathsf{O}_{32})( \mathsf{O}_{23}\mathsf{O}_{45} + \mathsf{O}_{24}\mathsf{O}_{35}+\mathsf{O}_{25}\mathsf{O}_{43}) =0.\]

The monomial $\mathsf{O}_{13}\mathsf{O}_{24}^2\mathsf{O}_{35}$ occurs only in the support of the equation $\mathsf{O}_{1234}\mathsf{O}_{2345}=0$. Similarly $\mathsf{O}_{i, i+2}\mathsf{O}_{i+1,i+3}^2\mathsf{O}_{i+2, i+4}$ occurs uniquely in the equation $\mathsf{O}_{i,i+1,i+2,i+3}\mathsf{O}_{i+1,i+2,i+3,i+4}$. This accounts for five equations which for pictographical reasons we call the `A' equation. We write these five equations as the five rotations of the following equation.
\begin{align}\label{eq:Arels} -\mathsf{O}_{13}\mathsf{O}_{24}^2\mathsf{O}_{35} =& \mathsf{O}_{12}\mathsf{O}_{23}\mathsf{O}_{34}\mathsf{O}_{45} + \mathsf{O}_{12}\mathsf{O}_{24}\mathsf{O}_{34}\mathsf{O}_{35} - \mathsf{O}_{12}\mathsf{O}_{25}\mathsf{O}_{34}^2 + \mathsf{O}_{13}\mathsf{O}_{23}\mathsf{O}_{24}\mathsf{O}_{45}\\ &+ \mathsf{O}_{13}\mathsf{O}_{24}\mathsf{O}_{25}\mathsf{O}_{43} - \mathsf{O}_{14}\mathsf{O}_{23}^2\mathsf{O}_{45}+ \mathsf{O}_{14}\mathsf{O}_{23}\mathsf{O}_{24}\mathsf{O}_{35} + \mathsf{O}_{14}\mathsf{O}_{23}\mathsf{O}_{25}\mathsf{O}_{43} \nonumber\end{align}
These five equations account for $A,B$ subsets of size $4$ such that $A \cap B$ is a subset of size $3$ of consecutive (modulo $5$) numbers.

Now for the other $5$ tableau relations, when $A \cap B$ is not three consecutive numbers. We get
\[ \mathsf{O}_{1234}\mathsf{O}_{1345}=  (\mathsf{O}_{12}\mathsf{O}_{34} +\mathsf{O}_{13}\mathsf{O}_{24} + \mathsf{O}_{14}\mathsf{O}_{32})( \mathsf{O}_{13}\mathsf{O}_{45} + \mathsf{O}_{14}\mathsf{O}_{35}+\mathsf{O}_{15}\mathsf{O}_{43})=0 \]
and the monomial $\mathsf{O}_{13}\mathsf{O}_{14}\mathsf{O}_{24}\mathsf{O}_{35}$ is uniquely in the support of the equation $\mathsf{O}_{1234}\mathsf{O}_{1345}=0$. Similarly, the monomial $\mathsf{O}_{i ,i+2} \mathsf{O}_{i ,i+3} \mathsf{O}_{i+1,i+3}\mathsf{O}_{i+2,i+4}$ is uniquely in the support of the equation $\mathsf{O}_{i,i+1,i+2.i+3}\mathsf{O}_{i,i+2,i+3,i+4}$. We thus have five more equations which are the rotations of the following equation

\begin{align}\label{eq:starrels} -\mathsf{O}_{13}\mathsf{O}_{14}\mathsf{O}_{24}\mathsf{O}_{35} =& \mathsf{O}_{12}\mathsf{O}_{13}\mathsf{O}_{34}\mathsf{O}_{45} + \mathsf{O}_{12}\mathsf{O}_{14}\mathsf{O}_{34}\mathsf{O}_{35} - \mathsf{O}_{12}\mathsf{O}_{15}\mathsf{O}_{34}^2 + \mathsf{O}_{13}\mathsf{O}_{13}\mathsf{O}_{24}\mathsf{O}_{45}\\ &+ \mathsf{O}_{13}\mathsf{O}_{15}\mathsf{O}_{24}\mathsf{O}_{43} - \mathsf{O}_{14}\mathsf{O}_{13}\mathsf{O}_{24}\mathsf{O}_{45}+ \mathsf{O}_{14}\mathsf{O}_{23}\mathsf{O}_{14}\mathsf{O}_{35} + \mathsf{O}_{14}\mathsf{O}_{15}\mathsf{O}_{23}\mathsf{O}_{43} \nonumber\end{align}

We call these equations the uncapped star relations. We can use equation (\ref{eq:Arels}) and (\ref{eq:starrels}) to remove any instance of $\mathsf{O}_{13}\mathsf{O}_{24}^2\mathsf{O}_{35}$ and $\mathsf{O}_{13}\mathsf{O}_{14}\mathsf{O}_{24}\mathsf{O}_{35}$ or rotations of such monomials. 

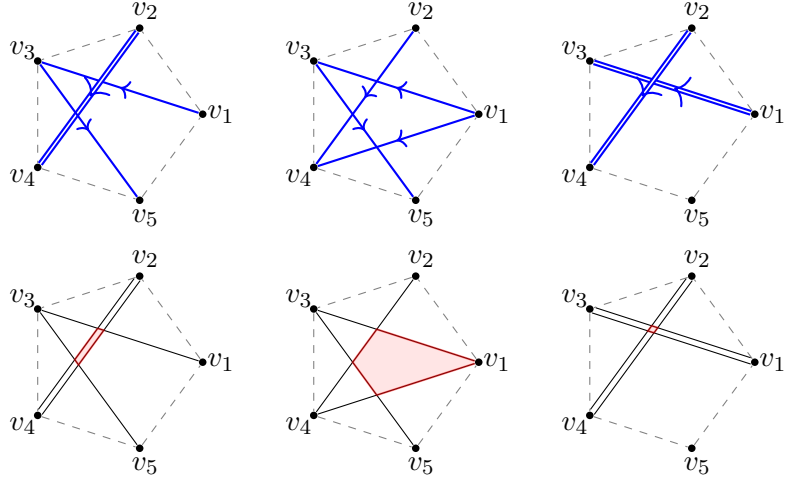
\begin{figure}[ht]
\centering
  \begin{diagram}{5}{
    \draw[midarrow] (v1) -- (v3);
     \draw[midarrow,double] (v2) -- (v4);
    \draw[midarrow] (v3) -- (v5);
  }
  \end{diagram} \quad 
  \begin{diagram}{5}{
    \draw[midarrow] (v1) -- (v3);
    \draw[midarrow] (v1) -- (v4);
    \draw[midarrow] (v2) -- (v4);
    \draw[midarrow] (v3) -- (v5);
  }
   \path[name intersections={of=A and B, by=I1}];
  \path[name intersections={of=B and C, by=I2}];
  \path[name intersections={of=C and A, by=I3}];
   \draw[red, thick, fill=red!20, opacity=0.5]
    (v1) -- (I1) -- (I2) -- (I3) -- cycle;
  \end{diagram}\quad 
   \begin{diagram}{5}{
    \draw[midarrow,double] (v1) -- (v3);
    \draw[midarrow,double] (v2) -- (v4);
  }
  \end{diagram}\\
  
  \begin{tikzpicture}[
    baseline=(current bounding box.center),
    scale=1.2,
    every node/.style={circle, fill=black, inner sep=1pt},
    midarrow/.style={
      blue,
      thick,
      postaction={decorate},
      decoration={
        markings,
        mark=at position 0.5 with {\arrow{>}}
      }
    }]
  \foreach \i in {1,...,5} {
      \node (v\i) at ({360/5 * (\i - 1)}:1) {};
      \node[draw=none, fill=none] at ({360/5 * (\i - 1)}:1.2) {\(v_{\i}\)};
    }

    \foreach \i [evaluate=\i as \j using {int(mod(\i,5)+1)}] in {1,...,5} {
      \draw[dashed, gray] (v\i) -- (v\j);
    }
     \node[draw=none, fill=none] (v21) at ({360/5 * (2 - 1)+2}:1) {};
     \node[draw=none, fill=none] (v41) at ({360/5 * (4 - 1)-2}:1) {};
     \node[draw=none, fill=none] (v29) at ({360/5 * (2 - 1)-2}:1) {};
     \node[draw=none, fill=none] (v49) at ({360/5 * (4 - 1)+2}:1) {};
    \draw[name path=A] (v1) -- (v3);
    \draw[name path=B] (v21) -- (v41);
    \draw[name path=C] (v29) -- (v49);
    \draw[name path=D] (v3) -- (v5);

   \path[name intersections={of=A and B, by=I1}];
   \path[name intersections={of=D and B, by=I2}];
   \path[name intersections={of=D and C, by=I3}];
   \path[name intersections={of=A and C, by=I4}];
  
   \draw[red, thick, fill=red!20, opacity=0.5]
    (I1) -- (I2) -- (I3) -- (I4) -- (I1);
 \end{tikzpicture} \quad \begin{tikzpicture}[
    baseline=(current bounding box.center),
    scale=1.2,
    every node/.style={circle, fill=black, inner sep=1pt},
    midarrow/.style={
      blue,
      thick,
      postaction={decorate},
      decoration={
        markings,
        mark=at position 0.5 with {\arrow{>}}
      }
    }]
  \foreach \i in {1,...,5} {
      \node (v\i) at ({360/5 * (\i - 1)}:1) {};
      \node[draw=none, fill=none] at ({360/5 * (\i - 1)}:1.2) {\(v_{\i}\)};
    }

    \foreach \i [evaluate=\i as \j using {int(mod(\i,5)+1)}] in {1,...,5} {
      \draw[dashed, gray] (v\i) -- (v\j);
    }

    \draw[name path=A] (v1) -- (v3);
    \draw[name path=B] (v1) -- (v4);
    \draw[name path=C] (v2) -- (v4);
    \draw[name path=D] (v3) -- (v5);

   \path[name intersections={of=A and C, by=I1}];
   \path[name intersections={of=C and D, by=I2}];
   \path[name intersections={of=B and D, by=I3}];
  
   \draw[red, thick, fill=red!20, opacity=0.5]
    (v1) -- (I1) -- (I2) -- (I3) -- (v1);
 \end{tikzpicture} \quad \begin{tikzpicture}[
    baseline=(current bounding box.center),
    scale=1.2,
    every node/.style={circle, fill=black, inner sep=1pt},
    midarrow/.style={
      blue,
      thick,
      postaction={decorate},
      decoration={
        markings,
        mark=at position 0.5 with {\arrow{>}}
      }
    }]
  \foreach \i in {1,...,5} {
      \node  (v\i) at ({360/5 * (\i - 1)}:1) {};
      \node[draw=none, fill=none] at ({360/5 * (\i - 1)}:1.2) {\(v_{\i}\)};
    }

    \foreach \i [evaluate=\i as \j using {int(mod(\i,5)+1)}] in {1,...,5} {
      \draw[dashed, gray] (v\i) -- (v\j);
    }

      \node [draw=none, fill=none] (v11) at ({360/5 * (1 - 1)+2}:1) {};
     \node [draw=none, fill=none] (v31) at ({360/5 * (3 - 1)-2}:1) {};
        \node [draw=none, fill=none] (v19) at ({360/5 * (1 - 1)-2}:1) {};
     \node [draw=none, fill=none] (v39) at ({360/5 * (3 - 1)+2}:1) {};

     \node [draw=none, fill=none] (v21) at ({360/5 * (2 - 1)+2}:1) {};
     \node [draw=none, fill=none] (v41) at ({360/5 * (4 - 1)-2}:1) {};
      \node [draw=none, fill=none] (v29) at ({360/5 * (2 - 1)-2}:1) {};
     \node [draw=none, fill=none] (v49) at ({360/5 * (4 - 1)+2}:1) {};

    \draw[name path=A] (v19) -- (v39);
    \draw[name path=B] (v11) -- (v31);
    \draw[name path=C] (v29) -- (v49);
    \draw[name path=D] (v21) -- (v41);

   \path[name intersections={of=A and C, by=I1}];
   \path[name intersections={of=A and D, by=I2}];
   \path[name intersections={of=B and D, by=I3}];
   \path[name intersections={of=B and C, by=I4}];
  
   \draw[red, thick, fill=red!20, opacity=0.5]
    (I1) -- (I2) -- (I3) -- (I4) -- (I1);
 \end{tikzpicture}
  \caption{Diagrams associated to the monomials $\mathsf{O}_{13}\mathsf{O}_{24}^2\mathsf{O}_{35}$, $\mathsf{O}_{13}\mathsf{O}_{14}\mathsf{O}_{24}\mathsf{O}_{35}$ and $\mathsf{O}_{24}^2 \mathsf{O}_{13}^2$. One can see why we call the relation involving the left diagram an 'A' relation and the middle an uncapped star relation, while the right diagram corresponds to a double cross relation. }
  \label{fig:AnduncappedStardiags} 
  \end{figure}

Let us define an internal quadrilateral of a diagram $D$ to be a quadrilateral following parts of the chords of $D$ (including double lines as separate lines) that has at least $3$ corners that are not vertices of the diagram. Note that the diagrams associated to $\mathsf{O}_{13}\mathsf{O}_{24}^2\mathsf{O}_{35}$, $\mathsf{O}_{13}\mathsf{O}_{14}\mathsf{O}_{24}\mathsf{O}_{35}$ and $\mathsf{O}_{13}^2\mathsf{O}_{24}^2$ have a single internal quadrilateral and the 15 diagrams involved in the `A' relations (\ref{eq:Arels}) uncapped star relations  (\ref{eq:starrels}) and double cross relations (\ref{eq:doublecrossingrel}) are the only diagrams with $4$ chords that have an internal quadrilateral. 

Now consider any monomial $M_O$ in $\mathsf{O}_{ij}$. This monomials corresponds to a unique diagram $D_{M_O}$. Suppose that this diagram has $k>0$ internal quadrilaterals. Then we can apply relations (\ref{eq:doublecrossingrel}), (\ref{eq:Arels}) or (\ref{eq:starrels}) to write $M_O$ as a sum of monomials each with associated diagrams with at most $k-1$ internal quadrilaterals. Applying this successively, we have deduced that every monomial in $\mathsf{O}_{ij}$ can be written as a span of monomials associated to diagrams with no internal quadrilaterals. To finish the proof note that any diagram in $5$ vertices that has zero internal quadrilaterals can only have $3$ internal chords and if $m_{i,i+2} > 1$ then $m_{i+1,i+3} + m_{i+1,1+4} <1$. This is precisely the condition for a diagram to be uncrossable (\ref{eq:uncross5cond}). Therefore the uncrossable diagrams and those with no internal quadrilateral precisely agree and we find that the monomials corresponding to uncrossable diagrams span the subalgebra generated by $\mathsf{O}_{ij}$.
\end{proof}

\bibliographystyle{abbrv}

\bibliography{References}

\end{document}